\documentclass[11pt, a4paper]{amsart}
\pdfoutput=1 
\usepackage[dvips]{graphicx}
\usepackage{amssymb, amstext, amscd, amsmath, color}
\usepackage{epstopdf}
\usepackage{tikz,mathdots,enumerate}
\usepackage{pgfplots}
\usepackage{url}

\usetikzlibrary{arrows}

\definecolor{cof}{RGB}{219,144,71}
\definecolor{pur}{RGB}{186,146,162}
\definecolor{greeo}{RGB}{91,173,69}
\definecolor{greet}{RGB}{52,111,72}

\usepackage{hhline}
\usepackage{hyperref}

\begin{document}

\newtheorem{theorem}{Theorem}[section]
\newtheorem{corollary}[theorem]{Corollary}
\newtheorem{proposition}[theorem]{Proposition}
\newtheorem{lemma}[theorem]{Lemma}

\theoremstyle{definition}
\newtheorem{remark}[theorem]{Remark}
\newtheorem{definition}[theorem]{Definition}
\newtheorem{example}[theorem]{Example}
\newtheorem{conjecture}[theorem]{Conjecture}

\newcommand{\FFock}{\mathcal{F}}
\newcommand{\kil}{\mathsf{k}}
\newcommand{\Hil}{\mathsf{H}}
\newcommand{\hil}{\mathsf{h}}
\newcommand{\Kil}{\mathsf{K}}
\newcommand{\Real}{\mathbb{R}}
\newcommand{\Rplus}{\Real_+}

\newcommand{\bC}{{\mathbb{C}}}
\newcommand{\bD}{{\mathbb{D}}}
\newcommand{\bN}{{\mathbb{N}}}
\newcommand{\bQ}{{\mathbb{Q}}}
\newcommand{\bR}{{\mathbb{R}}}
\newcommand{\bT}{{\mathbb{T}}}
\newcommand{\bX}{{\mathbb{X}}}
\newcommand{\bZ}{{\mathbb{Z}}}
\newcommand{\bH}{{\mathbb{H}}}
\newcommand{\BH}{{\B(\H)}}
\newcommand{\bsl}{\setminus}
\newcommand{\ca}{\mathrm{C}^*}
\newcommand{\cstar}{\mathrm{C}^*}
\newcommand{\cenv}{\mathrm{C}^*_{\text{env}}}
\newcommand{\rip}{\rangle}
\newcommand{\ol}{\overline}
\newcommand{\td}{\widetilde}
\newcommand{\wh}{\widehat}
\newcommand{\sot}{\textsc{sot}}
\newcommand{\wot}{\textsc{wot}}
\newcommand{\wotclos}[1]{\ol{#1}^{\textsc{wot}}}
 \newcommand{\A}{{\mathcal{A}}}
 \newcommand{\B}{{\mathcal{B}}}
 \newcommand{\C}{{\mathcal{C}}}
 \newcommand{\D}{{\mathcal{D}}}
 \newcommand{\E}{{\mathcal{E}}}
 \newcommand{\F}{{\mathcal{F}}}
 \newcommand{\G}{{\mathcal{G}}}
\renewcommand{\H}{{\mathcal{H}}}
 \newcommand{\I}{{\mathcal{I}}}
 \newcommand{\J}{{\mathcal{J}}}
 \newcommand{\K}{{\mathcal{K}}}
\renewcommand{\L}{{\mathcal{L}}}
 \newcommand{\M}{{\mathcal{M}}}
 \newcommand{\N}{{\mathcal{N}}}
\renewcommand{\O}{{\mathcal{O}}}
\renewcommand{\P}{{\mathcal{P}}}
 \newcommand{\Q}{{\mathcal{Q}}}
 \newcommand{\R}{{\mathcal{R}}}
\renewcommand{\S}{{\mathcal{S}}}
 \newcommand{\T}{{\mathcal{T}}}
 \newcommand{\U}{{\mathcal{U}}}
 \newcommand{\V}{{\mathcal{V}}}
 \newcommand{\W}{{\mathcal{W}}}
 \newcommand{\X}{{\mathcal{X}}}
 \newcommand{\Y}{{\mathcal{Y}}}
 \newcommand{\Z}{{\mathcal{Z}}}

\newcommand{\supp}{\operatorname{supp}}
\newcommand{\conv}{\operatorname{conv}}
\newcommand{\cone}{\operatorname{cone}}
\newcommand{\vspan}{\operatorname{span}}
\newcommand{\proj}{\operatorname{proj}}
\newcommand{\sgn}{\operatorname{sgn}}
\newcommand{\rank}{\operatorname{rank}}
\newcommand{\Isom}{\operatorname{Isom}}
\newcommand{\qIsom}{\operatorname{q-Isom}}
\newcommand{\Cknet}{{\mathcal{C}_{\text{knet}}}}
\newcommand{\Ckag}{{\mathcal{C}_{\text{kag}}}}
\newcommand{\rind}{\operatorname{r-ind}}
\newcommand{\lind}{\operatorname{r-ind}}
\newcommand{\ind}{\operatorname{ind}}
\newcommand{\coker}{\operatorname{coker}}
\newcommand{\Aut}{\operatorname{Aut}}
\newcommand{\Hom}{\operatorname{Hom}}
\newcommand{\GL}{\operatorname{GL}}
\newcommand{\tr}{\operatorname{tr}}

\newcommand{\eqnwithbr}[2]{%
\refstepcounter{equation}
\begin{trivlist}
\item[]#1 \hfill $\displaystyle #2$ \hfill (\theequation)
\end{trivlist}}

\setcounter{tocdepth}{1}

\title[Maxwell-Laman counts for bar-joint frameworks in normed spaces]{Maxwell-Laman counts for bar-joint frameworks in normed spaces}

\author[D. Kitson and  B. Schulze]{D. Kitson and  B. Schulze}
\thanks{The first named author is supported by EPSRC grant  EP/J008648/1.} 
\email{d.kitson@lancaster.ac.uk, b.schulze@lancaster.ac.uk}
\address{Dept.\ Math.\ Stats.\\ Lancaster University\\
Lancaster LA1 4YF \\U.K. }

\subjclass[2010]{52C25, 20C35, 05C50}
\keywords{rigidity matrix, bar-joint framework,  infinitesimal rigidity, Minkowski space, symmetric framework}

\begin{abstract}
 The rigidity matrix is a fundamental tool for studying the infinitesimal rigidity properties of Euclidean bar-joint frameworks. In this paper we generalize this tool and introduce a rigidity matrix for  bar-joint frameworks in  arbitrary finite dimensional real normed vector spaces. Using this new matrix, we derive necessary Maxwell-Laman-type counting conditions for a well-positioned bar-joint framework in a real normed vector space to be infinitesimally rigid.
Moreover, we derive symmetry-extended  counting conditions for a bar-joint framework with a non-trivial symmetry group to be isostatic (i.e., minimally infinitesimally rigid). These conditions imply very simply stated restrictions on the number of those structural components that are fixed by the various symmetry operations of the framework. Finally, we offer some observations and conjectures regarding combinatorial characterisations of symmetric, isostatic bar-joint frameworks  in $(\bR^2,\|\cdot\|_\P)$, where the unit ball $\P$ is a quadrilateral.
\end{abstract}

\maketitle


\section{Introduction.}

In 1864, J.C. Maxwell formulated a necessary counting condition for a bar-joint framework to be (infinitesimally) rigid in $(\mathbb{R}^d,\|\cdot \|_2)$ \cite{maxwell}.
Building on Maxwell's observation, G. Laman established the first combinatorial characterisation of   rigid bar-joint frameworks which are generically placed in the Euclidean plane in 1970, thereby launching the field of combinatorial rigidity \cite{Lamanbib}. Several further equivalent characterisations of  generic rigid bar-joint frameworks in $(\mathbb{R}^2,\|\cdot \|_2)$  have been established since then (see \cite{lovyem,tay}, for example). A fundamental tool for proving these results is the rigidity matrix, whose rank, row dependencies and column dependencies completely describe  the infinitesimal rigidity properties of a framework.
Combinatorial characterisations of generic rigid bar-joint frameworks in higher dimensions have not yet been found. However, there exist significant partial results for the special classes of body-bar, body-hinge, and molecular frameworks \cite{Lamanbib}.

Very recent work has considered the infinitesimal rigidity of bar-joint frameworks in some non-Euclidean normed spaces. Specifically, for the  $\ell^q$ norms, $1\leq q \leq \infty$, $q\neq 2$, and for the polyhedral norms,  analogues of Laman's theorem were established in \cite{kit-pow} and \cite{kitson}. In this paper, we take a more general viewpoint and consider the infinitesimal rigidity properties of bar-joint frameworks in an arbitrary finite dimensional real normed vector space (also referred to in the literature as a Minkowski space \cite{tho}). In particular, in Section~\ref{sec:basicdef} we introduce a rigidity matrix and, using this new matrix, we derive Maxwell-Laman-type counting conditions which are necessary for a bar-joint framework to be infinitesimally rigid.

Over the last few years, a range of tools and methods have been developed for analysing the impact of symmetry on the rigidity properties of frameworks in Euclidean $d$-space (see e.g. \cite{cfgsw,FGsymmax,KG2,owen,schulze,schtan}). In particular, in \cite{FGsymmax} Fowler and Guest derived new necessary conditions for a symmetric bar-joint framework to be isostatic (i.e., minimally infinitesimally rigid) in $(\mathbb{R}^d,\|\cdot \|_2)$, and it was shown in \cite{cfgsw} that these conditions can be stated in a very simple form in terms of the number of structural components that are fixed by various symmetry operations of the framework (see also \cite{gsw,sfg} for extensions of these results  to body-bar and body-hinge frameworks). The fundamental underlying result is that
the rigidity matrix of a symmetric framework can be transformed into a block-decomposed form using methods from group representation theory \cite{KG2,BS2}. For the symmetry groups of order $2$ and $3$, it was shown in \cite{schulze,BS4} that Laman's conditions, together with the added conditions on the number of fixed  structural components, are also sufficient for a bar-joint framework which is symmetry-generic (i.e., as generic as possible subject to the given symmetry constraints) to be isostatic in the Euclidean plane. However, the analogous questions for the remaining symmetry groups which allow isostatic frameworks in $\mathbb{R}^2$ remain open.

In Section~\ref{sec:symmetric} we extend the results in \cite{cfgsw,FGsymmax} and derive new necessary criteria for a bar-joint framework $(G,p)$  with a non-trivial point group $\Gamma$ in a general Minkowski space to be isostatic.
Fundamental to this approach is Proposition~\ref{IntertwiningLemma} which shows that the rigidity matrix of $(G,p)$  intertwines representations of $\Gamma$ associated with the edges and vertices of $G$ (also known as the `internal' and `external' representation in the engineering community \cite{FGsymmax,KG2}). Subsequently, in Section~\ref{sec:FiniteIsometryGroup}, we follow the approach of Connelly et al. in \cite{cfgsw} to derive a complete list of the necessary counting conditions for  symmetric isostatic bar-joint frameworks in $2$- or $3$-dimensional normed spaces for which the group of linear isometries is finite. 
 As in the Euclidean case, these conditions are in terms of counts for the number of vertices and edges that are fixed by various symmetry operations. Using the results of Section~\ref{sec:symmetric}, analogous necessary conditions for isostaticity can also be obtained for symmetric frameworks in higher dimensions, and we provide a sample of those in Section~\ref{sec:FiniteIsometryGroup} as well.

Finally, in Section~\ref{sec:furwork}, we provide a number of observations and conjectures regarding both necessary and sufficient conditions for the existence of a (symmetric) well-positioned isostatic bar-joint framework  in $2$-dimensional normed spaces. In particular, we offer Laman-type conjectures for all possible symmetry groups for the polyhedral norm $\|\cdot\|_\P$ on $\bR^2$, where the unit ball $\P$ is a quadrilateral.


\section{Maxwell counts for infinitesimally rigid frameworks with general norms.} 
\label{sec:basicdef}
Let $X$ be a finite dimensional real vector space.
A {\em bar-joint framework}  in $X$ is a pair $(G,p)$ consisting of a finite simple graph $G$ and a point $p=(p(v))_{v\in V}\in X^{|V|}$ with the property that the components $p(v)$ are distinct points in $X$.
Here $V$ denotes the vertex set of $G$. The edge set of $G$ is denoted by $E$.

\subsection{The rigidity matrix  in general normed spaces.}
Let $\|\cdot\|$ be a norm on $X$ and denote by $S$ the unit sphere, $S=\{x\in X:\|x\|=1\}$.
Recall that the norm is said to be {\em smooth} at a point $x_0\in S$ if there exists exactly one supporting hyperplane for $S$ at $x_0$. Equivalently, there exists a unique linear functional $f\in X^*$ (called a {\em support functional} for $x_0$) such that $|f(x)|\leq1$ for all $x\in S$ and $f(x_0)=1$.  Note that in this case the norm is also smooth at $-x_0$ with unique support functional $-f$.
For each pair $x_0,y\in X$ define,
\[\psi_-(x_0;y):=\lim_{t\to 0^-}\frac{1}{t}\left(\|x_0+ty\|-\|x_0\|\right), \,\,\,\,\,\,\,\,\,\,\,\,
\psi_+(x_0;y):=\lim_{t\to 0^+}\frac{1}{t}\left(\|x_0+ty\|-\|x_0\|\right).\]
As the norm is necessarily a convex function both of these one-sided limits exist. Moreover, if $x_0\in S$ and $f$ is a support functional for $x_0$ then $\psi_-(x_0;y)\leq f(y)\leq\psi_+(x_0;y)$ for all $y\in X$.
Define $\psi(x_0;y)$ to be the two-sided limit $\psi(x_0;y):=\lim_{t\to 0}\frac{1}{t}\left(\|x_0+ty\|-\|x_0\|\right)$, if this limit exists.
We require the following well-known fact and include a proof for the readers convenience
(see also \cite{gil,phe} for example).

\begin{lemma}
\label{Smooth}
Let $x_0\in S$.
Then  the norm is smooth at $x_0$ if and only if $\psi(x_0;y)$ exists for all $y\in X$.
\end{lemma}

\proof
Suppose  the norm is smooth at $x_0$ and that $\psi_-(x_0;y_0)\not=\psi_+(x_0;y_0)$ for some $y_0\in X$. By an application of the Hahn-Banach theorem,
for each $\lambda\in \bR$ with $\psi_-(x_0;y_0)<\lambda<\psi_+(x_0;y_0)$ the mapping $h_\lambda:\bR y_0 \to \bR$, $h_\lambda(ty_0)=t\lambda$ extends to a linear functional $h_\lambda\in X^*$ such that $\psi_-(x_0;y)\leq h_\lambda(y) \leq\psi_+(x_0;y)$ for all $y\in X$. Here we use the fact that $\psi_+(x_0;\cdot)$ is a sublinear function and $\psi_-(x_0;y)=-\psi_+(x_0;-y)$ for all $y\in X$. 
It follows that $h_\lambda$ is a support functional for $x_0$. However, this contradicts the uniqueness of the support functional at $x_0$ and so $\psi(x_0;y)$ must exist for all $y\in X$.
Conversely, if  $\psi(x_0;y)$ exists for all $y\in X$ then it follows from the sublinearity of $\psi_+(x_0;\cdot)$ and the identity $\psi_-(x_0;y)=-\psi_+(x_0;-y)$ that the map $f(y)=\psi(x_0;y)$ is a linear functional on $X$.
Moreover, $f$ is a support functional for $x_0$ and, since every support functional $g$ for $x_0$ must satisfy $\psi_-(x_0;y)\leq g(y)\leq\psi_+(x_0;y)$ for all $y\in X$, the support functional $f$ must be unique. Thus the norm is smooth at $x_0$.
\endproof

\begin{definition}
Let $(G,p)$ be a bar-joint framework in $X$ and let $\|\cdot\|$ be a norm on $X$.
\begin{enumerate}[(i)]
\item
An edge $vw$ of $G$ is said to be {\em well-positioned} (for  the placement $p$ and norm $\|\cdot\|$)
if the norm is smooth at $\frac{p_v-p_w}{\|p_v-p_w\|}$.
\item
The bar-joint framework $(G,p)$ is said to be {\em well-positioned} in $(X,\|\cdot\|)$ if every edge of $G$ is well-positioned.
\end{enumerate}
\end{definition}

Given a well-positioned edge $e=vw$ in $G$, the unique support functional for $\frac{p_v-p_w}{\|p_v-p_w\|}$
is denoted $\varphi_{v,w}$. Note that $\varphi_{v,w}=-\varphi_{w,v}$ and from the proof of Lemma \ref{Smooth},
\begin{eqnarray}
\label{SupportEqn}
\varphi_{v,w}(y) = \psi\left(\frac{p_v-p_w}{\|p_v-p_w\|};y\right) = \psi\left(p_v-p_w;y\right)
\end{eqnarray}
for all $y\in X$.

\begin{definition}
The {\em rigidity matrix} for a well-positioned bar-joint framework $(G,p)$ is an $|E|\times |V|$ matrix $R(G,p)$ with entries in the dual space $X^*$ given by,
\[a_{(e,v)}=\left\{\begin{array}{ll}
\varphi_{v,w} & \mbox{ if }e=vw \mbox{ for some vertex } w,\\
0 & \mbox{ otherwise,}
\end{array}\right.\]
for all $(e,v)\in E\times V$.
\end{definition}
The rigidity matrix $R(G,p)$ may be viewed as a linear map from $X^{|V|}\to \bR^{|E|}$
given by the formula
\begin{eqnarray}
\label{RigidityMatrixEqn}
\left(u_v\right)_{v\in V} \, \mapsto \, \left(\,\sum_{v\in V} a_{(e,v)}(u_v)\,\right)_{e\in E}
= \,\,\,\,\,\left(\, \varphi_{v,w}(u_v-u_w)\,\right)_{vw\in E}
\end{eqnarray}

\begin{remark}
In computations it is sometimes more natural to define the entries of the rigidity matrix to be the support functionals for $p_v-p_w$ rather than for the normalised vectors $\frac{p_v-p_w}{\|p_v-p_w\|}$.
This is common practice in the case of the Euclidean norm and is also the approach taken in \cite{kit-pow} for the smooth $\ell^p$ norms. In \cite{kitson} the above formulation of the rigidity matrix  is used in the context of polyhedral norms. The approach taken here means that the rigidity matrix coincides with the differential of the rigidity map which we now introduce.
\end{remark}

\begin{definition}
Let $G$ be a simple graph and let $\|\cdot\|$ be a norm on $X$. The {\em rigidity map}
for $G$ is the mapping
\[f_G:X^{|V|}\to \bR^{|E|}, \,\,\,\,\,\,\,\, x=\left(x_v\right)_{v\in V}\mapsto \left(\|x_v-x_w\|\right)_{vw\in E}\]
\end{definition}

\begin{proposition}
\label{Differentiable}
A bar-joint framework $(G,p)$ is well-positioned in $(X,\|\cdot\|)$ if and only if
the rigidity map $f_G$ is differentiable at $p$.
Moreover, in this case $df_G(p)=R(G,p)$.
\end{proposition}

\proof
If $(G,p)$ is well-positioned then, by Formula (\ref{SupportEqn}), for each $u=(u_v)_{v\in V}\in X^{|V|}$  the directional derivative of $f_G$ at $p$ in the direction of $u$ exists and satisfies,
\begin{eqnarray*}
\label{DirectionalDerivativeEqn}
D_uf_G(p) := \lim_{t\to 0}\,\frac{f_G(p+tu)-f_G(p)}{t}
=\left(\,\psi(p_v-p_w;\,u_v-u_w)\,\right)_{vw\in E}
= \left(\, \varphi_{v,w}(u_v-u_w)\,\right)_{vw\in E} 
\end{eqnarray*}
Thus by Formula (\ref{RigidityMatrixEqn}),  $D_uf_G(p) = R(G,p)u$. 
In particular, the map $u\mapsto D_uf_G(p)$ is linear and, since $f_G$ is convex, it follows that $f_G$ is differentiable at $p$ with differential $df_G(p)$ satisfying,
\begin{eqnarray}
\label{DirectionalDerivativeEqn}
df_G(p)u = D_uf_G(p)= R(G,p)u, \,\,\,\,\,\,\, \mbox{ for all }u\in X^{|V|}.
\end{eqnarray}
Conversely,  if the rigidity map $f_G$ is differentiable at $p$ then the directional derivative $D_uf_G(p)$ exists for all $u\in X^{|V|}$.
Hence if $vw\in E$ is an edge of $G$ then $\psi(p_v-p_w;y)$ exists for all $y\in X$ and so, by Lemma \ref{Smooth}, the norm is smooth at $\frac{p_v-p_w}{\|p_v-p_w\|}$. We conclude that  $(G,p)$ is well-positioned in $(X,\|\cdot\|)$.
\endproof

\begin{remark}
If $\|\cdot\|$ is a smooth norm (such as the Euclidean norm on $\bR^d$ or, more generally, an $\ell^p$ norm on $\bR^d$ with $p\in(1,\infty)$) then  all bar-joint frameworks in $(X,\|\cdot\|)$ are well-positioned.
If $\|\cdot\|$ is a polyhedral norm then $(G,p)$ is well-positioned if and only if  $p_v-p_w$ is contained in the interior of the conical hull of some facet of the unit ball,
for each edge $vw\in E$.
\end{remark}

\begin{definition}
Let $(G,p)$ be a bar-joint framework in  $(X,\|\cdot\|)$.
An {\em infinitesimal flex} of $(G,p)$ is an element $u\in X^{|V|}$ which satisfies
$D_uf_G(p)=0$.
\end{definition}

The collection of all infinitesimal flexes of $(G,p)$ is denoted $\F(G,p)$.
By  Formula (\ref{DirectionalDerivativeEqn}) it is clear that if $(G,p)$ is well-positioned then $\F(G,p)=\ker R(G,p)$.

\begin{definition}
A {\em rigid motion} of $(X,\|\cdot\|)$ is a family of continuous maps \[\alpha_x:[-1,1]\to X, \,\,\,\,\,\,\,\,x\in X\]
such that $\alpha_x(t)$ is differentiable at $t=0$ with $\alpha_x(0)=x$ and $\|\alpha_x(t)-\alpha_y(t)\| =\|x-y\|$ for all pairs $x,y\in X$ and all $t\in [-1,1]$.
\end{definition}

If $\{\alpha_x:x\in X\}$ is a rigid motion of $(X,\|\cdot\|)$ then
$(\alpha_{p(v)}'(0))_{v\in V}\in X^{|V|}$ is an infinitesimal flex of $(G,p)$ (see \cite[Lemma 2.1]{kit-pow}).
We regard such infinitesimal flexes as trivial and denote by $\T(G,p)$ the collection of all trivial infinitesimal flexes of $(G,p)$.
 If every infinitesimal flex of $(G,p)$ is trivial then we say that $(G,p)$ is {\em infinitesimally rigid}. 
If $(G,p)$ is infinitesimally rigid and the removal of any edge results in a framework which is not infinitesimally rigid,  then we say that $(G,p)$ is {\em isostatic}.

\begin{theorem}
\label{thm:maxwell}
Let $(G,p)$ be a well-positioned bar-joint framework in $(X,\|\cdot\|)$.
\begin{enumerate}
\item If $(G,p)$ is infinitesimally rigid then,
$|E|\geq \dim (X)\,|V|- \dim \T(G,p)$.
 \item If $(G,p)$ is isostatic then
$|E|= \dim (X)\,|V|-\dim \T(G,p)$.
\item If $(G,p)$ is isostatic and $H$ is a subgraph of $G$ then,
$|E(H)|\leq \dim (X)\,|V(H)|-\dim \T(H,p)$.
\end{enumerate}
\end{theorem}

\proof
$(i)$
In general, the rigidity matrix satisfies,
\[\dim \ker R(G,p) - \dim \coker R(G,p) = \dim (X^{|V|}) - \dim \bR^{|E|} =\dim (X)\,|V|  - |E|\]
Since $(G,p)$ is infinitesimally rigid and well-positioned in $(X,\|\cdot\|)$, $\dim \ker R(G,p)=\dim \F(G,p)
=\dim \T(G,p)$. Thus
\[|E|- \dim \coker R(G,p)  = \dim (X)\,|V|- \dim \T(G,p)\]

$(ii)$
If $(G,p)$ is isostatic then $R(G,p)$ is  row independent and $\dim \ker R(G,p)=\dim \T(G,p)$.
Thus,
\[|E| = \rank R(G,p) = \dim (X^{|V|}) - \dim \ker R(G,p) =  \dim (X)\,|V| - \dim \T(G,p).\]

$(iii)$
If $(G,p)$ is isostatic and $H$ is a subgraph of $G$ then $R(H,p)$ is row independent and so,
\[|E(H)| = \rank R(H,p) = \dim (X^{|V(H)|}) - \dim \ker R(H,p) \leq  \dim (X)|V(H)| - \dim \T(H,p).\]
\endproof

\section{Maxwell counts for symmetric frameworks with general norms.}
\label{sec:symmetric}
Let $(G,p)$ be a bar-joint framework in $X$ and let $\|\cdot\|$ be a norm on $X$.
An {\em automorphism} of the graph $G$ is a permutation of the vertices $\pi:V\to V$ such that $vw\in E$ if and only if $\pi(v)\pi(w)\in E$. The set $ \textrm{Aut}(G)$ of all automorphisms of $G$ is  a subgroup of the permutation group on $V$.
An {\em action} of a group $\Gamma$ on $G$ is a group homomorphism $\theta:\Gamma\to  \textrm{Aut}(G)$.
The graph $G$ is said to be {\em $\Gamma$-symmetric} (with respect to $\theta$) if there exists such an action.
As a notational convenience, if the action $\theta$ is clear from the context then we will denote $\theta(\gamma)(v)$ by  $\gamma v$ for each vertex $v$ and  $(\gamma v)(\gamma w)$ by  $\gamma(vw)$ for each  edge $vw$ of $G$.

Suppose there exists a group representation $\tau:\Gamma\rightarrow \GL(X)$  such that $\tau(\gamma)$ is an isometry of $(X,\|\cdot\|)$ for each $\gamma\in \Gamma$.
The framework $(G,p)$ is said to be {\em \textrm{$\Gamma$-symmetric}} (with respect to $\theta$ and $\tau$) if
\begin{equation}
\label{eq:symmetric_func}
\tau(\gamma) (p(v))=p(\gamma v) \qquad \text{for all } \gamma\in \Gamma \text{ and all } v\in V.
\end{equation}
Moreover, we say that $\gamma$ is a \emph{symmetry operation} and $\Gamma$ is a {\em symmetry group} of the framework $(G,p)$.

\subsection{Symmetry adapted Maxwell counts for isostatic frameworks.} 
\label{subsec:neccond}
Let $(G,p)$ be a $\Gamma$-symmetric framework with respect
to an action $\theta:\Gamma\rightarrow \textrm{Aut}(G)$ and a group representation $\tau:\Gamma\rightarrow \GL(X)$. Define a pair of  permutation representations of $\Gamma$ as follows,
\[P_V:\Gamma\to \GL(\bR^{|V|}), \,\,\,\,\,\,\,\, P_V(\gamma)(a_v)_{v\in V}= (a_{\gamma^{-1}v})_{v\in V},\]
\[P_E:\Gamma\to \GL(\bR^{|E|}), \,\,\,\,\,\,\,\, P_E(\gamma)(a_e)_{e\in E}= (a_{\gamma^{-1}e})_{e\in E}.\]
The trivial representation of $\Gamma$ on $X$ is denoted by ${\bf1}$.
We will require the following tensor product representations of $\Gamma$ on $X^{|V|}$,
\[{\bf1}\otimes P_V: \Gamma \to \GL(X^{|V|}), \,\,\,\,\,\,\,\,\, ({\bf1}\otimes P_V)(\gamma)(x_v)_{v\in V}=(x_{\gamma^{-1}v})_{v\in V}\]
\[\hspace{8mm} \tau\otimes P_V: \Gamma \to \GL(X^{|V|}), \,\,\,\,\,\,\,\,\, (\tau\otimes P_V)(\gamma)(x_v)_{v\in V}=(\tau(\gamma)(x_{\gamma^{-1}v}))_{v\in V}\]

Recall that given two representations $\rho_1:\Gamma \to \GL(X)$ and $\rho_2:\Gamma\to \GL(Y)$
with representation spaces $X$ and $Y$,
a linear map $T:X\rightarrow Y$ is said to be a \emph{$\Gamma$-linear map} of $\rho_1$ and $\rho_2$  if $T\circ\rho_1(\gamma)=\rho_2(\gamma)\circ T$ for all $\gamma \in \Gamma$.
The vector space of all $\Gamma$-linear maps of $\rho_1$ and $\rho_2$
is denoted by $\Hom_{\Gamma}(\rho_1,\rho_2)$.

\begin{proposition}
\label{IntertwiningLemma}
Let $(G,p)$ be a well-positioned bar-joint framework in $(X,\|\cdot\|)$.
If $(G,p)$ is $\Gamma$-symmetric with respect to $\theta:\Gamma\to \Aut(G)$ and $\tau:\Gamma\to \GL(X)$ then,
\[df_G(p)\in \Hom_{\Gamma}(\tau \otimes P_V,P_E).\]
\end{proposition}

\proof Since $(G,p)$ be well-positioned, by Proposition \ref{Differentiable}, the rigidity map $f_G$ is differentiable at $p$. Let $\gamma\in\Gamma$.
Then it is readily verified that $f_G$ is differentiable at $({\bf1}\otimes P_V)(\gamma)p=(p(\gamma^{-1}v))_{v\in V}$ with differential,
\[df_G( ({\bf1}\otimes P_V)(\gamma)p) = P_E(\gamma) \circ df_G(p)\circ ({\bf1}\otimes P_V)(\gamma)^{-1}.  \]
Moreover, since $\tau(\gamma)$ is a linear isometry of $X$ a similar verification shows that
$f_G$ is  differentiable at $(\tau\otimes P_V)(\gamma)p$ with differential,
\[df_G( (\tau\otimes P_V)(\gamma)p) = P_E(\gamma) \circ df_G(p)\circ (\tau\otimes P_V)(\gamma)^{-1}.  \]
From Formula (\ref{eq:symmetric_func}),  $\tau(\gamma) (p(v))=p(\gamma v)$ for all $v\in V$ 
and so $p=(\tau(\gamma)p(\gamma^{-1}v))_{v\in V}=(\tau\otimes P_V)(\gamma)p$. Thus,
\begin{eqnarray*}
df_G(p)\circ(\tau\otimes P_V)(\gamma)
&=& df_G((\tau\otimes P_V)(\gamma)p)\circ(\tau\otimes P_V)(\gamma) \\
&=& P_E(\gamma) \circ df_G(p).
\end{eqnarray*}

\endproof

Recall that if $\rho:\Gamma\to\GL(X)$ is a representation of  $\Gamma$ with representation space $X$ then a  subspace $Y$ of $X$ is said to be \emph{$\rho$-invariant} if $\rho(\gamma)(Y)\subseteq Y$ for all $\gamma\in \Gamma$.

\begin{proposition}
Let $(G,p)$ be a bar-joint framework in $(X,\|\cdot\|)$
which is $\Gamma$-symmetric with respect to $\theta$ and $\tau$. Then
$\T(G,p)$ is a $\tau \otimes P_V$-invariant subspace of $X^{|V|}$.
\end{proposition}

\proof
Let $\gamma\in \Gamma$ and let $u\in \T(G,p)$.
Then there exists a rigid motion $\{\alpha_x:x\in X\}$ of $(X,\|\cdot\|)$ with $u(v)=\alpha_{p(v)}'(0)$ for all $v\in V$.
Now let $\beta_x(t) = \tau(\gamma)(\alpha_{\tau(\gamma)^{-1}x}(t))$.
Since $\tau(\gamma)$ is a linear isometry we have,
\begin{eqnarray*}
\|\beta_x(t)-\beta_y(t)\| &=& \| \tau(\gamma)(\alpha_{\tau(\gamma)^{-1}x}(t))
- \tau(\gamma)(\alpha_{\tau(\gamma)^{-1}y}(t))\| \\
&=&\| \alpha_{\tau(\gamma)^{-1}x}(t)
- \alpha_{\tau(\gamma)^{-1}y}(t)\|\\
&=& \|\tau(\gamma)^{-1}x-\tau(\gamma)^{-1}y\| \\
&=& \|x-y\|
\end{eqnarray*}
Also, since $\tau(\gamma) (p(v))=p(\gamma v)$ for each $v\in V$ we have,
\[\beta_{p(v)}'(0) = \tau(\gamma)\alpha_{\tau(\gamma)^{-1}p(v)}'(0)
= \tau(\gamma)u(\gamma^{-1}v)\]
Hence $\{\beta_x:x\in X\}$ is a rigid motion of $(X,\|\cdot\|)$ which satisfies,
\[(\beta'_{p(v)}(0))_{v\in V} = (\tau\otimes P_V)(\gamma)u\]
We conclude that $(\tau\otimes P_V)(\gamma)u\in \T(G,p)$ and so $\T(G,p)$ is  $\tau \otimes P_V$-invariant.
\endproof

We denote by $(\tau \otimes P_V)^{(\T)}$ the subrepresentation  of $\tau \otimes P_V$ with representation space $\T(G,p)$.
Recall that the \emph{character} of a representation $\rho:\Gamma \to \GL(X)$ is the row vector $\chi(\rho)$ whose $i$-th component is the trace of $\rho(\gamma_i)$, for some fixed ordering $\gamma_1,\ldots, \gamma_{|\Gamma|}$ of the elements of $\Gamma$.

\begin{theorem}
\label{maxwellsrulewithchar} Let $(G,p)$ be a well-positioned bar-joint framework in $(X,\|\cdot\|)$
which is $\Gamma$-symmetric with respect to $\theta$ and $\tau$.
If  $(G,p)$ is isostatic then,
\begin{equation}
\label{maxchar}
\chi(P_E) =\chi(\tau \otimes P_V) - \chi((\tau \otimes P_V)^{(\T)}).
\end{equation}
\end{theorem}

\proof
By Maschke's Theorem, $\T(G,p)$ has a $\tau \otimes P_V$-invariant complement $Q$ in $X^{|V|}$. We may therefore form the subrepresentation $(\tau \otimes P_V)^{(Q)}$ of $\tau \otimes P_V$ with representation space $Q$.
Since  $(G,p)$ is isostatic, the restriction of the differential $df_G(p)$ to $Q$ is an isomorphism onto $\bR^{|E|}$.
Moreover, since  $df_G(p)$ is $\Gamma$-linear with respect to the representations
$\tau \otimes P_V$ and $P_E$, this restriction is $\Gamma$-linear for the representations
$(\tau \otimes P_V)^{(Q)}$ and $P_E$.
Hence $(\tau \otimes P_V)^{(Q)}$ and $P_E$ are isomorphic representations of $\Gamma$.
We conclude that,
\[\chi(P_E) =\chi( (\tau \otimes P_V)^{(Q)})
=\chi(\tau \otimes P_V) - \chi((\tau \otimes P_V)^{(\T)}).\]
\endproof

Let $\theta:\Gamma\to\Aut(G)$ be a group action on $G$.
A vertex $v$ of $G$ is said to be {\em fixed} by $\gamma\in \Gamma$ (with respect to $\theta$) if $\gamma v =v$. Similarly, an edge $e=vw$ of $G$ is said to be fixed by $\gamma\in \Gamma$ (with respect to $\theta$) if $\gamma e =e$, i.e., if either $\gamma v =v$ and $\gamma w =w$, or, $\gamma v =w$ and $\gamma w =v$.
The sets of vertices and edges of a $\Gamma$-symmetric graph $G$ which are fixed by $\gamma\in \Gamma$ are denoted by $V_{\gamma}$ and  $E_{\gamma}$, respectively.

\begin{corollary}
\label{maxwellcor}
Let $(G,p)$ be a well-positioned bar-joint framework in $(X,\|\cdot\|)$
which is $\Gamma$-symmetric with respect to $\theta$ and $\tau$.
If  $(G,p)$ is isostatic then for each $\gamma\in \Gamma$,
\begin{eqnarray}|E_\gamma| = \tr(\tau(\gamma))\,|V_\gamma|-\tr((\tau \otimes P_V)^{(\T)}(\gamma)).\end{eqnarray}
\end{corollary}

\proof
Note that $\tr(P_V(\gamma)) = |V_\gamma|$  and $\tr(P_E(\gamma))=|E_\gamma|$
for each $\gamma\in\Gamma$.
The result now follows from Theorem \ref{maxwellsrulewithchar} and the fact that $\tr((\tau\otimes P_V)(\gamma))=\tr(\tau(\gamma))\tr(P_V(\gamma))$.
\endproof

\section{Normed spaces with finitely many linear isometries.}
\label{sec:FiniteIsometryGroup}
In many cases (such as for the non-Euclidean $\ell^p$ norms and polyhedral norms) the group of linear isometries of a normed vector space is a finite group.
In this section we present necessary counting conditions for isostatic bar-joint frameworks in such normed spaces and in the presence of a variety of  symmetry operations.
We use the Schoenflies notation for the symmetry groups $\Gamma$ and  symmetry operations $\gamma\in \Gamma$ considered in this section, as this is one of the standard notations in the literature about symmetric structures \cite{bishop,FGsymmax,KG2,schulze}.

\begin{proposition}
\label{FiniteIsometryGroup}
Let $(G,p)$ be a well-positioned bar-joint framework in $(X,\|\cdot\|)$
which is isostatic and $\Gamma$-symmetric with respect to $\theta$ and $\tau$.
If the group of linear isometries of $(X,\|\cdot\|)$ is finite then,
\begin{enumerate}[(i)]
\item $|E| = \dim (X)\,(|V|-1)$, and,
\item $|E_\gamma| = \tr(\tau(\gamma))\,(|V_\gamma|-1)$ for each $\gamma\in \Gamma$.
\end{enumerate}
\end{proposition}

\proof
If the group of linear isometries of $(X,\|\cdot\|)$ is finite then, by \cite[Lemma 2.5]{kit-pow},
$\T(G,p)$ consists of translational infinitesimal flexes only, i.e. $\T(G,p)=\{(a,\ldots,a)\in X^{|V|}:a\in X\}$.
Thus $(\tau \otimes P_V)^{(\T)}(\gamma)=(\tau\otimes {\bf 1})(\gamma)$ for each $\gamma\in \Gamma$ and so $\tr((\tau \otimes P_V)^{(\T)}(\gamma))=\tr(\tau(\gamma))$. The result now follows from Corollary
\ref{maxwellcor}.
\endproof

In Tables~\ref{tab:2D} and \ref{tab:3D}, we present character tables for the representations appearing in Formula (\ref{maxchar}). These tables apply respectively to frameworks in  two-dimensional and three-dimensional normed spaces for which the group of linear isometries is finite. The symmetry operations for frameworks in these non-Euclidean spaces are defined below.

\begin{table}[htp]\begin{center}
    \begin{tabular}{l|c c c c}
                        & $\phantom{-}Id$       &  $C_{n > 2}$   &
                            $\phantom{-}C_2$   & $\phantom{-}s$  \\ \hline\hline
    $\chi(P_E)$& $|E|$     & $|E_n|$             &
                            $|E_2|$   & $|E_s|$\\\hline
    $\chi(\tau\otimes P_V)$&   $2|V|$     & $(2\cos \frac{2\pi}{n})|V_n|$       &
                            $-2|V_2|$    & $0$       \\ \hline
    $\chi((\tau\otimes P_V)^{(\T)})$
                        & $2$      & $2\cos \frac{2\pi}{n}$  &
                            $-2$& $0$       \\
        \end{tabular}
    \caption{Calculations of characters for the symmetry-extended counting rule for isostatic frameworks in a two-dimensional non-Euclidean space.}
    \label{tab:2D}
    \end{center}
\end{table}

\begin{table}[htp]\begin{center}
    \begin{tabular}{l|c c c c c c}
                        & $\phantom{-}Id$       &  $C_{n > 2}$   &
                            $\phantom{-}C_2$   & $\phantom{-}s$ & $i$ & $S_{n>2}$ \\ \hline\hline
    $\chi(P_E)$& $|E|$     & $|E_n|$             &
                            $|E_2|$   & $|E_s|$  & $|E_i|$   & $|E_{S_n}|$ \\\hline
    $\chi(\tau\otimes P_V)$&   $3|V|$     & $(2\cos \frac{2\pi}{n}+1)|V_n|$       &
                            $-|V_2|$    & $|V_s|$    & $-3|V_i|$ &   $(2\cos \frac{2\pi}{n}-1)|V_{S_n}|$   \\ \hline
    $\chi((\tau\otimes P_V)^{(\T)})$
                        & $3$      & $2\cos \frac{2\pi}{n}+1$  &
                            $-1$ & $1$  & $-3$ &  $2\cos \frac{2\pi}{n}-1$     \\
        \end{tabular}
    \caption{Calculations of characters for the symmetry-extended counting rule for isostatic frameworks in a three-dimensional non-Euclidean space.}
    \label{tab:3D}
    \end{center}
\end{table}

\subsection{Reflections and inversions.}
Let $(G,p)$ be a bar-joint framework in $X$ which is $\Gamma$-symmetric with respect to $\theta$ and $\tau$.
A symmetry operation $s\in \Gamma$ is called a {\em reflection} if $\tau(s)=I-2P$ where $P$ is a rank one projection on $X$. A symmetry group which is generated by a reflection $s$ is denoted by $\mathcal{C}_{s}$.

\begin{corollary}
\label{Reflection}
Let $(G,p)$ be a well-positioned bar-joint framework in $(X,\|\cdot\|)$
which is isostatic and $\Gamma$-symmetric with respect to $\theta$ and $\tau$.
Suppose the group of linear isometries of $(X,\|\cdot\|)$ is finite.
If the symmetry group $\Gamma$ contains a reflection $s\in \Gamma$ then,
\begin{enumerate}[(i)]
\item $|E_s|=(\dim(X)-2)\,(|V_s|-1)$.
\item If $\dim(X)=2$ then $|E_s|=0$.
\item If $\dim(X)\geq3$ then the following two conditions hold,
\begin{enumerate}[(a)]
\item $|V_s|\geq1$, and,
\item $|V_s|=1$ if and only if $|E_s|=0$.
\end{enumerate}
\end{enumerate}
\end{corollary}

\proof
Note that $\tau(s) = I-2P$ where  $P$ is a rank one projection and so $\tr(\tau(s))=\dim(X)-2$.
Thus $(i)$ follows from Proposition \ref{FiniteIsometryGroup}.
The remaining conditions are simple consequences of $(i)$.
\endproof

A symmetry operation $i\in\Gamma$ is called an {\em inversion} if $\tau(i)=-I$.

\begin{corollary}
Let $(G,p)$ be a well-positioned bar-joint framework in $(X,\|\cdot\|)$
which is isostatic and $\Gamma$-symmetric with respect to $\theta$ and $\tau$.
Suppose the group of linear isometries of $(X,\|\cdot\|)$ is finite.
If the symmetry group $\Gamma$ contains an inversion  $i\in \Gamma$ then
one of the following two conditions holds,
\begin{enumerate}[(i)]
\item $|V_i| =0$, and, $|E_i|=\dim (X)$.
\item $|V_i| =1$, and, $|E_i|=0$.
\end{enumerate}
\end{corollary}

\proof
Note that $\tau(i)$ has trace $-\dim(X)$ and so, by Proposition \ref{FiniteIsometryGroup},
$|E_i| =  -\dim (X)\,(|V_i|-1)$. The result follows.
\endproof

\subsection{Half-turn and $4$-fold rotations.}
A symmetry operation $C_n\in\Gamma$ is called an {\em $n$-fold rotation} ($n \geq 2$) if there exists a two-dimensional subspace $Y$ of $X$ with a complementary space $Z$ (which could be $0$) such that $\tau(C_n)=S\oplus I_Z$ where
$S:Y\to Y$ has matrix representation $${\left(\begin{array}{cc} \cos(2\pi /n) &  -\sin(2\pi /n) \\  \sin(2\pi /n) &  \cos(2\pi /n)\end{array}\right)}$$
with respect to some basis for $Y$ and $I_Z$ is the identity operator on $Z$.
In the case $n=2$, $C_2$ is also called a {\em half-turn rotation}.
A symmetry group which is generated by an $n$-fold rotation $C_n$ is denoted by $\mathcal{C}_{n}$.
In the following, $|V_n|$, and $|E_n|$ denote respectively the numbers of vertices and edges that are fixed by an $n$-fold rotation $C_n$.

\begin{corollary}
\label{2Fold}
Let $(G,p)$ be a well-positioned bar-joint framework in $(X,\|\cdot\|)$
which is isostatic and $\Gamma$-symmetric with respect to $\theta$ and $\tau$.
Suppose the group of linear isometries of $(X,\|\cdot\|)$ is finite.
If the symmetry group $\Gamma$ contains a half-turn rotation $C_2\in \Gamma$ then,
\begin{enumerate}[(i)]
\item $|E_2| =  \left(\dim (X)-4\right)\,(|V_2|-1)$.
\item If $\dim (X)=2$, then one of the following two conditions holds,
\begin{enumerate}[(a)]
\item  $|V_2| =0$, and, $|E_2|=2$.
\item $|V_2| =1$, and, $|E_2|=0$.
\end{enumerate}
\item If  $\dim (X)=3$, then one of the following two conditions holds,
\begin{enumerate}[(a)]
\item  $|V_2| =0$, and, $|E_2|=1$.
\item $|V_2| =1$, and, $|E_2|=0$.
\end{enumerate}
\item If  $\dim (X)=4$, then  $|E_2|=0$.
\item If  $\dim (X)\geq5$, then the following two conditions hold,
\begin{enumerate}[(a)]
\item  $|V_2| \geq1$, and,
\item $|V_2| =1$ if and only if $|E_2|=0$.
\end{enumerate}
\end{enumerate}
\end{corollary}

\begin{proof}
The trace of $\tau(C_2)$ is $\dim(X)-4$ and so $(i)$ follows from Proposition \ref{FiniteIsometryGroup}. The remaining parts are simple consequences of $(i)$.
\end{proof}

The following corollary presents necessary counting conditions in the presence of a four-fold rotation $C_4$.

\begin{corollary}
\label{4Fold}
Let $(G,p)$ be a well-positioned bar-joint framework in $(X,\|\cdot\|)$
which is isostatic and $\Gamma$-symmetric with respect to $\theta$ and $\tau$.
Suppose the group of linear isometries of $(X,\|\cdot\|)$ is finite.
If the symmetry group $\Gamma$ contains a  four-fold rotation $C_4\in \Gamma$ then,
\begin{enumerate}[(i)]
\item $|E_4| =  \left(\dim (X)-2\right)\,(|V_4|-1)$.
\item If  $\dim (X)=2$, then  $|V_4|\leq 1$ and $|E_4|=0$.
\item If  $\dim (X)=3$ or $\dim (X)=4$, then $|V_4|=1$ and $|E_4|=0$.
\item If  $\dim (X)\geq5$, then the following two conditions hold,
\begin{enumerate}[(a)]
\item  $|V_4| \geq1$, and,
\item $|V_4| =1$ if and only if $|E_4|=0$.
\end{enumerate}
\end{enumerate}
\end{corollary}

\proof
The trace of $\tau(C_4)$ is $\dim(X)-2$ and so $(i)$ follows from Proposition \ref{FiniteIsometryGroup}. If $\dim(X)=2$ then by $(i)$, $|E_4|= 0$.
Note that $(C_4)^2$ is a half-turn rotation and so,
by Corollary \ref{2Fold}$(ii)$, at most one vertex of $G$ is fixed by $(C_4)^2$.
A vertex which is fixed by $C_4$ must also be fixed by $(C_4)^2$ and so  $|V_4|\leq 1$.
This proves $(ii)$.
If $\dim (X)=3$ then by $(i)$, $|V_4|\geq 1$. Suppose $v$ and $w$ are vertices of $G$ which are fixed by $C_4$. Then $v$ and $w$ are also fixed by $(C_4)^2$.
By Corollary \ref{2Fold}$(iii)$, there exists at most one vertex of $G$ which is fixed by $(C_4)^2$.
Thus $v=w$ and so $|V_4|=1$. Also, $|E_4|=0$ by $(i)$. 
If $\dim (X)=4$ and  $vw$ is an edge of $G$ which is fixed by $C_4$ then $vw$ is also fixed by $(C_4)^2$.
However, this contradicts Corollary \ref{2Fold}$(iv)$, and so no edge of $G$ is fixed by $C_4$.
Thus $|E_4|=0$ and, by $(i)$, $|V_4|=1$.  This establishes $(iii)$.
Part $(iv)$ follows directly from $(i)$.
\endproof

\begin{remark}
Note that if a framework $(G,p)$ is $\Gamma$-symmetric and a vertex $v$ is fixed by a non-trivial symmetry operation $\gamma$ in $\Gamma$,  then $p(v)$ must occupy a special geometric position in the space $X$. Specifically, $p(v)$ must be lie in the kernel of $I-\tau(\gamma)$. For example, in dimension $2$, if a vertex $v$ is fixed by a half-turn $C_2$ then $p(v)$ must lie on the origin, and if $v$ is fixed by a reflection $s$ then $p(v)$ must lie on the corresponding mirror line. Similarly, if an edge $vw$ is fixed by a  symmetry operation $\gamma$ then $p(v)-p(w)$ must lie in the kernel of either $I-\tau(\gamma)$ or $I+\tau(\gamma)$. 
For example, if $vw$ is fixed by a half-turn $C_2$, then the corresponding bar must be centred at the origin, and if $vw$ is fixed by a reflection, then this bar must either lie along the mirror line, or, in the complementary line along which the reflection acts and centred at the mirror line.

Analogously, in the $3$-dimensional case if a vertex $v$ is fixed by an $n$-fold rotation $C_n$ then $p(v)$ must lie on the rotation axis, and if $v$ is fixed by a reflection $s$ then $p(v)$ must lie on the corresponding mirror plane.  If an edge $vw$ is fixed by a
half-turn, then the corresponding bar must lie either along the rotation axis, or, in the plane of rotation and centred at the axis. Similarly, if $vw$ is fixed by a reflection then this bar must lie either within the mirror plane, or, in the complementary line along which the reflection acts and centred at the mirror plane.
If the edge $vw$ is fixed by an $n$-fold  rotation $C_{n>2}$ then the corresponding bar must lie along the rotation axis. 

In any dimension, if a vertex $v$ is fixed by an inversion $i$ then $p(v)$ must lie on the origin and  if an edge $vw$ is fixed by an inversion then the corresponding bar must be centred at the origin.
\end{remark}

The following example considers the class of $\ell^p$ norms with $p\not=2$ and demonstrates  several $2$-dimensional isostatic frameworks for all possible symmetry groups. 

\begin{example} \label{ex:2dexamp}
Consider a bar-joint framework $(G,p)$ in $\bR^2$ which is $\Gamma$-symmetric with respect to an $\ell^p$ norm, with $1\leq p\leq \infty$ and $p\not=2$.
The group of linear isometries for these norms has order $8$ and is generated by the reflections in the lines $x=0$ and $x=y$.
Thus, each non-trivial symmetry operation for $(G,p)$ must be either a reflection $s$,
a half-turn rotation $C_2$ or a four-fold rotation $C_4$.
For this reason the only symmetry groups we need consider are the reflectional group $\mathcal{C}_s$, the rotational groups $\mathcal{C}_2$ and $\mathcal{C}_4$, and the dihedral groups $\mathcal{C}_{2v}$ and $\mathcal{C}_{4v}$.

In Figure \ref{fig:henstart}, several examples of isostatic bar-joint frameworks in $(\mathbb{R}^2,\|\cdot\|_\infty)$ are illustrated for these symmetry groups. 
The framework $(a)$ illustrates reflectional symmetry in the $y$-axis with one fixed vertex and no fixed edges while framework $(b)$ illustrates reflectional symmetry in the line $y=x$ with two fixed vertices and no fixed edges (cf. Corollary \ref{Reflection}). The framework $(c)$ illustrates symmetry under a half-turn rotation about the origin with no fixed vertices and two fixed edges (cf. Corollary \ref{2Fold}).
Framework $(d)$ illustrates symmetry under four-fold rotation about the origin with one fixed vertex and no fixed edges (cf. Corollary \ref{4Fold}).
The framework $(e)$ illustrates symmetry under the dihedral group $\C_{2v}$ generated by a reflection in the $x$-axis and a half-turn rotation about the origin. In this case, exactly one vertex and no edges are fixed by each non-trivial symmetry operation in $\C_{2v}$.
The framework $(f)$ illustrates symmetry under the dihedral group $\C_{4v}$ generated by a reflection in the $x$-axis and a four-fold rotation about the origin. Again, exactly one vertex and no edges are fixed by each non-trivial symmetry operation in $\C_{4v}$.

Note that each of the frameworks in Figure~\ref{fig:henstart} can easily be checked to be isostatic using the results in \cite{kit-pow}. As shown in \cite{kit-pow}, a well-positioned framework $(G, p)$ in $(\mathbb{R}^2,\|\cdot\|_\infty)$ induces a natural labelling of the edges of $G$ (with two colours) determined by the placement $p$ relative to the facets of the unit ball of the norm $\|\cdot\|_\infty$, and $(G,p)$ is isostatic in $(\mathbb{R}^2,\|\cdot\|_\infty)$ if and only if the two induced monochrome subgraphs are spanning trees. The induced spanning trees of the underlying graphs of the frameworks in Figure~\ref{fig:henstart} are indicated in black and gray, respectively.
\end{example}

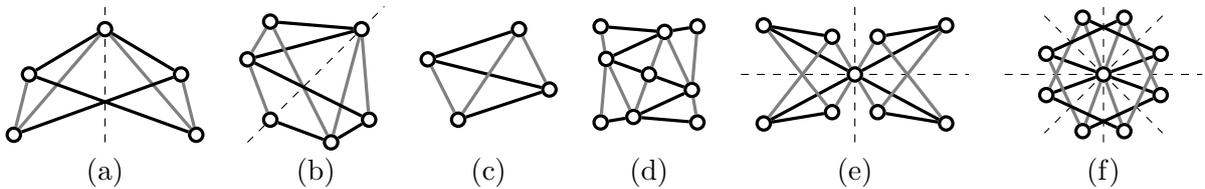
\begin{figure}[htp]
\begin{center}
\begin{tikzpicture}[very thick,scale=1]
\tikzstyle{every node}=[circle, draw=black, fill=white, inner sep=0pt, minimum width=5pt];

\draw[dashed,thin] (0,-0.9)  --  (0,0.9);

\path (0,0.6) node (p1)  {} ;
\path (-1,0) node (p2)  {} ;
\path (1,0) node (p3)  {} ;
\path (-1.2,-0.8) node (p4)  {} ;
\path (1.2,-0.8) node (p6)  {} ;

\draw (p1)  --  (p2);
\draw (p1)  --  (p3);
\draw[gray] (p1)  --  (p4);
\draw[gray] (p1)  --  (p6);
\draw[gray] (p2)  --  (p4);
\draw[gray] (p3)  --  (p6);
\draw (p4)  --  (p3);
\draw (p6)  --  (p2);
\node [rectangle, draw=white, fill=white] (b) at (0,-1.3) {(a)};
\end{tikzpicture}
\hspace{0.2cm}
 \begin{tikzpicture}[very thick,scale=1]
\tikzstyle{every node}=[circle, draw=black, fill=white, inner sep=0pt, minimum width=5pt];

\draw[dashed,thin] (-0.9,-0.9)  --  (0.9,0.9);

\path (0.6,0.6) node (p1)  {} ;
\path (-0.6,0.7) node (p2)  {} ;
\path (-0.9,0.2) node (p3)  {} ;
\path (0.7,-0.6) node (p4)  {} ;
\path (0.2,-0.9) node (p5)  {} ;
\path (-0.6,-0.6) node (p6)  {} ;

\draw (p1)  --  (p2);
\draw (p1)  --  (p3);
\draw [gray](p1)  --  (p4);
\draw[gray] (p1)  --  (p5);
\draw[gray] (p2)  --  (p3);
\draw (p4)  --  (p5);
\draw[gray] (p2)  --  (p5);
\draw (p4)  --  (p3);
\draw[gray] (p6)  --  (p3);
\draw (p6)  --  (p5);
\node [rectangle, draw=white, fill=white] (b) at (0,-1.3) {(b)};
\end{tikzpicture}
\hspace{0.2cm}
\begin{tikzpicture}[very thick,scale=1]
\tikzstyle{every node}=[circle, draw=black, fill=white, inner sep=0pt, minimum width=5pt];

\path (-0.8,0.2) node (p1)  {} ;
\path (0.8,-0.2) node (p2)  {} ;
\path (-0.4,-0.6) node (p3)  {} ;
\path (0.4,0.6) node (p4)  {} ;

\draw (p1)  --  (p2);
\draw[gray] (p1)  --  (p3);
\draw (p1)  --  (p4);
\draw[gray] (p2)  --  (p4);
\draw[gray] (p3)  --  (p4);
\draw (p3)  --  (p2);
\node [rectangle, draw=white, fill=white] (b) at (0,-1.3) {(c)};
\end{tikzpicture}
 \hspace{0.2cm}
\begin{tikzpicture}[very thick,scale=1]
\tikzstyle{every node}=[circle, draw=black, fill=white, inner sep=0pt, minimum width=5pt];

\path (70:0.6cm) node (p1)  {} ;
\path (160:0.6cm) node (p2)  {} ;
\path (250:0.6cm) node (p3)  {} ;
\path (340:0.6cm) node (p4)  {} ;
\path (70:0cm) node (c)  {} ;

\path (45:0.9cm) node (o1)  {} ;
\path (135:0.9cm) node (o2)  {} ;
\path (225:0.9cm) node (o3)  {} ;
\path (315:0.9cm) node (o4)  {} ;

\draw (p1)  --  (p2);
\draw[gray] (p2)  --  (p3);
\draw (p3)  --  (p4);
\draw[gray] (p1)  --  (p4);
\draw[gray] (c)  --  (p1);
\draw (c)  --  (p2);
\draw[gray] (c)  --  (p3);
\draw (c)  --  (p4);

\draw (o1)  --  (p1);
\draw[gray] (o1)  --  (p4);

\draw (o2)  --  (p1);
\draw[gray] (o2)  --  (p2);

\draw[gray]  (o3)  --  (p2);
\draw(o3)  --  (p3);

\draw (o4)  --  (p3);
\draw[gray] (o4)  --  (p4);

\node [rectangle, draw=white, fill=white] (b) at (0,-1.3) {(d)};
\end{tikzpicture}
\hspace{0.2cm}
\begin{tikzpicture}[very thick,scale=1]
\tikzstyle{every node}=[circle, draw=black, fill=white, inner sep=0pt, minimum width=5pt];
\draw[dashed,thin] (0,-0.9)  --  (0,0.9);
\draw[dashed,thin] (-1.5,0)  --  (1.5,0);

\path (0,0) node (p1)  {} ;
\path (-0.3,0.5) node (p2)  {} ;
\path (0.3,0.5) node (p22)  {} ;
\path (-0.3,-0.5) node (p3)  {} ;
\path (0.3,-0.5) node (p33)  {} ;

\path (-1.2,0.65) node (p4)  {} ;
\path (1.2,0.65) node (p44)  {} ;
\path (-1.2,-0.65) node (p5)  {} ;
\path (1.2,-0.65) node (p55)  {} ;

\draw[gray] (p1)  --  (p2);
\draw[gray] (p1)  --  (p22);
\draw[gray] (p1)  --  (p3);
\draw[gray] (p1)  --  (p33);
\draw (p1)  --  (p4);
\draw (p1)  --  (p44);
\draw (p1)  --  (p5);
\draw (p1)  --  (p55);

\draw[gray] (p2)  --  (p5);
\draw[gray] (p3)  --  (p4);

\draw (p3)  --  (p5);
\draw (p2)  --  (p4);

\draw[gray] (p22)  --  (p55);
\draw[gray] (p33)  --  (p44);

\draw (p33)  --  (p55);
\draw (p22)  --  (p44);

\node [rectangle, draw=white, fill=white] (b) at (0,-1.3) {(e)};
\end{tikzpicture}
\hspace{0.2cm}
\begin{tikzpicture}[very thick,scale=1]
\tikzstyle{every node}=[circle, draw=black, fill=white, inner sep=0pt, minimum width=5pt];
\draw[dashed,thin] (0,-0.9)  --  (0,0.9);
\draw[dashed,thin] (-1.3,0)  --  (1.3,0);
\draw[dashed,thin] (45:1.1cm)  --  (225:1.1cm);
\draw[dashed,thin] (135:1.1cm)  --  (315:1.1cm);

\path (0,0) node (p1)  {} ;
\path (70:0.8cm) node (p2)  {} ;
\path (110:0.8cm) node (p22)  {} ;
\path (160:0.8cm) node (p3)  {} ;
\path (200:0.8cm) node (p33)  {} ;

\path (250:0.8cm) node (p4)  {} ;
\path (290:0.8cm) node (p44)  {} ;
\path (340:0.8cm) node (p5)  {} ;
\path (20:0.8cm) node (p55)  {} ;

\draw[gray] (p1)  --  (p2);
\draw[gray] (p1)  --  (p22);
\draw (p1)  --  (p3);
\draw (p1)  --  (p33);
\draw[gray] (p1)  --  (p4);
\draw[gray](p1)  --  (p44);
\draw (p1)  --  (p5);
\draw (p1)  --  (p55);

\draw[gray] (p22)  --  (p33);
\draw[gray] (p3)  --  (p4);

\draw (p33)  --  (p44);
\draw (p5)  --  (p4);

\draw[gray] (p44)  --  (p55);
\draw[gray] (p5)  --  (p2);

\draw (p55)  --  (p22);
\draw (p2)  --  (p3);

\node [rectangle, draw=white, fill=white] (b) at (0,-1.3) {(f)};
\end{tikzpicture}
     \end{center}
     \caption{Examples of symmetric isostatic  bar-joint frameworks in $(\bR^2,\|\cdot\|_\infty)$: (a), (b) $\mathcal{C}_s$-symmetry with $|V_s|=1$ and $|V_s|=2$, respectively; (c) $\mathcal{C}_2$-symmetry (with $|V_2|=0$ and $|E_2|=2$); (d) $\mathcal{C}_4$-symmetry (with $|V_4|=1$); (e) $\mathcal{C}_{2v}$-symmetry;
     (f) $\mathcal{C}_{4v}$-symmetry.}
\label{fig:henstart}
\end{figure}

\subsection{$n$-fold Rotations, $n=3$ and $n\geq5$.}
The following corollary considers the case of a three-fold rotation.

\begin{corollary}
\label{3Fold}
Let $(G,p)$ be a well-positioned bar-joint framework in $(X,\|\cdot\|)$
which is isostatic and $\Gamma$-symmetric with respect to $\theta$ and $\tau$.
Suppose the group of linear isometries of $(X,\|\cdot\|)$ is finite.
If the symmetry group $\Gamma$ contains a  three-fold rotation   $C_3\in \Gamma$ 
 then,
\begin{enumerate}[(i)]
\item $|E_3| =  \left(\dim (X)-3\right)\,(|V_3|-1)$.
\item If $\dim (X)=2$, then one of the following two conditions holds,
\begin{enumerate}[(a)]
\item  $|V_3| =0$, and, $|E_3|=1$.
\item $|V_3| =1$, and, $|E_3|=0$.
\end{enumerate}
\item If  $\dim (X)=3$, then  $|E_3|=0$.
\item If  $\dim (X)\geq4$, then the following two conditions hold,
\begin{enumerate}[(a)]
\item  $|V_3| \geq1$, and,
\item $|V_3| =1$ if and only if $|E_3|=0$.
\end{enumerate}
\end{enumerate}
\end{corollary}

\begin{proof}
The trace of $\tau(C_3)$ is $\dim(X)-3$ and so $(i)$ follows from Proposition \ref{FiniteIsometryGroup}. The remaining parts are simple consequences of $(i)$.

\end{proof}

The case of a six-fold rotation is presented in the following corollary.

\begin{corollary}
Let $(G,p)$ be a well-positioned bar-joint framework in $(X,\|\cdot\|)$
which is isostatic and $\Gamma$-symmetric with respect to $\theta$ and $\tau$.
Suppose the group of linear isometries of $(X,\|\cdot\|)$ is finite.
If the symmetry group $\Gamma$ contains a  six-fold rotation   $C_6\in \Gamma$ 
 then,
\begin{enumerate}[(i)]
\item $|E_6| =  \left(\dim (X)-1\right)\,(|V_6|-1)$.
\item If $\dim (X)=2$ or $\dim(X)=3$, then $|V_6|=1$ and $|E_6|=0$.
\item  If $\dim (X)\geq4$ then the following two conditions hold,
\begin{enumerate}[(a)]
\item  $|V_6| \geq1$, and,
\item $|V_6| =1$ if and only if $|E_6|=0$.
\end{enumerate}
\end{enumerate}
\end{corollary}

\proof
The trace of $\tau(C_6)$ is $\dim(X)-1$ and so $(i)$ follows from Proposition \ref{FiniteIsometryGroup}.
If $\dim (X)=2$ or $\dim (X)=3$ then by $(i)$, $|V_6|\geq 1$. In particular, there exists at least one vertex of $G$ which is fixed by the $3$-fold rotation $(C_6)^2$.
By Corollary \ref{3Fold} parts $(ii)$ and $(iii)$, there are no edges of $G$ which are fixed by $(C_6)^2$. Hence $|E_6|=0$ and, by $(i)$, $|V_6|=1$.
This establishes $(ii)$. Part $(iii)$ follows directly from $(i)$.
\endproof

For all remaining values of $n$, the derived counting conditions for an $n$-fold rotation are presented below.

\begin{corollary}
Let $(G,p)$ be a well-positioned bar-joint framework in $(X,\|\cdot\|)$
which is isostatic and $\Gamma$-symmetric with respect to $\theta$ and $\tau$.
Suppose the group of linear isometries of $(X,\|\cdot\|)$ is finite.
If the symmetry group $\Gamma$ contains an $n$-fold rotation $C_n\in \Gamma$ 
with $n=5$, or, $n\geq 7$ then  $|V_n| =1$, and, $|E_n|=0$.
\end{corollary}

\proof
Note that the trace of $\tau(C_n)$ is $\dim(X)-2\cos\left(\frac{2\pi}{n}\right)$ and so by Proposition \ref{FiniteIsometryGroup}, $|E_n| =  \left(\dim (X)-2\cos\left(\frac{2\pi}{n}\right)\right)\,(|V_n|-1)$.
If $n=5$ or $n\geq 7$ then $2\cos\left(\frac{2\pi}{n}\right)\notin \bZ$ and so the result follows.
\endproof

\subsection{Improper rotations.}
Suppose $\dim(X)=3$.
An element $S_n\in\Gamma$ is called an {\em improper rotation}  by an angle of $\frac{2\pi}{n}$
if it is a composition of an $n$-fold rotation $C_n$, where $X=Y\oplus Z$ and $\tau(C_n)=S\oplus I_Z$, followed by a reflection $s$, where  $\tau(s)=I-2P$, such that $P$ is the rank one projection of $X$ along $Y$ onto $Z$. If $n=2$ then $S_2$ is an inversion and so we assume here that $n\geq3$.

\begin{corollary}
Let $(G,p)$ be a well-positioned bar-joint framework in $(X,\|\cdot\|)$
which is isostatic and $\Gamma$-symmetric with respect to $\theta$ and $\tau$.
Suppose $\dim (X)= 3$ and the group of linear isometries of $(X,\|\cdot\|)$ is finite.
If the symmetry group $\Gamma$ contains an improper rotation $S_n\in \Gamma$  then,
\begin{enumerate}[(i)]
\item $|E_{S_n}| =  \left(2\cos\left(\frac{2\pi}{n}\right) -1\right)\,(|V_{S_n}|-1)$.
\item  If  $n=3$, then $|V_{S_3}| =1$, and, $|E_{S_3}|=0$. 
\item  If $n=4$, then one of the following two conditions holds,
\begin{enumerate}[(a)]
\item  $|V_{S_4}| =0$, and, $|E_{S_4}|=1$.
\item $|V_{S_4}| =1$, and, $|E_{S_4}|=0$.
\end{enumerate}
\item  If $n=6$, then  $|E_{S_6}|=0$.
\item If $n=5$ or $n\geq 7$, then $|V_{S_n}|=1$ and $|E_{S_n}|=0$.
\end{enumerate}
\end{corollary}

\proof
Note that $\tau(S_n) = (I-2P)(S\oplus I_Z) = S\oplus (-I_Z)$ and so
$\tau(S_n)$ has trace $2\cos\left(\frac{2\pi}{n}\right) -(\dim (X)-2)$.
Part $(i)$ now follows from Proposition \ref{FiniteIsometryGroup}. To see $(ii)$, note that $(i)$ implies that either $|V_{S_3}| =0$ and $|E_{S_3}|=2$ or $|V_{S_3}| =1$ and $|E_{S_3}|=0$. However, $S_3^2$ is a three-fold rotation about the improper rotation axis of $S_3$, and by Corollary~\ref{3Fold}, we must have $|E_{3}|=0$. It follows that we must have $|V_{S_3}| =1$ and $|E_{S_3}|=0$. The remaining parts follow from $(i)$.
\endproof

\begin{remark}
If a framework $(G,p)$ has a symmetry group $\Gamma$ which contains an improper rotation $S_n = s\circ C_n$ and $n$ is odd then $\Gamma$ must also contain the reflection $s$ and the $n$-fold rotation $C_n$.
If $n$ is even then $\Gamma$ must contain the $\frac{n}{2}$-fold rotation $(C_n)^2$.
If a vertex $v$  is fixed by an improper rotation $S_{n>2}$ then $p(v)$ must lie on the origin and if an edge $vw$ is fixed by an improper rotation then the corresponding bar must lie along the improper rotation axis and centred at the mirror.
\end{remark}

The following example demonstrates a $3$-dimensional isostatic framework with symmetry group 
 $\mathcal{C}_{3h}$  generated by a reflection $s$ and a $3$-fold rotation $C_3$. The norm is a polyhedral norm with unit ball a hexagonal prism.

\begin{example}\label{ex:3dexamp}
Define a graph $G$ as follows. First, take four 3-cycles with respective vertices $\{v_1,v_2,v_3\}$, $\{v_1',v_2',v_3'\}$, $\{w_1,w_2,w_3\}$ and $\{w_1',w_2',w_3'\}$. Then adjoin the following additional edges: $v_jv_j'$ and $w_jw_j'$ for each $j$, $v_1w_3'$, $w_1v_3'$, $v_2w_1'$, $w_2v_1'$, $v_3w_2'$ and $w_3v_2'$. Finally, take the cone over this graph by adjoining a new vertex $o$ and an edge from $o$ to all other vertices. See also Figure~\ref{fig:3dexamp}.

Next, define a bar-joint framework $(G,p)$ in $\mathbb{R}^3$ by placing the vertices of $G$ as follows. Place $p(o)$ at the origin,  $p(v_1)$ at the point $(-1/4, -1, 1/4)$, and $p(v_1')$ at the point $(-1/4, -1, 3/2)$.
The placements for all other vertices are generated by taking the orbits of $p(v_1)$ and $p(v_1')$ under three-fold rotation about the $z$-axis and reflection in the $xy$-plane. In the Schoenflies notation, the  resulting framework then has the symmetry group $\mathcal{C}_{3h}$.

Consider the polyhedral norm on $\mathbb{R}^3$ with unit ball a hexagonal prism determined by the extreme points 
$(\cos(\pi(k-1)/3),\sin(\pi(k-1)/3), 1)$ and $(\cos(\pi(k-1)/3), \sin(\pi(k-1)/3), -1)$, for $k=1,2,...,6$.
It can now be verified (either through calculation of the rigidity matrix or by analysis of the induced monochrome subgraphs of $G$ \cite{kitson}) that $(G,p)$ is isostatic with respect to this polyhedral norm.

Note that $(G,p)$  satisfies the required Maxwell-Laman counts of Theorem~\ref{thm:maxwell}, as well as the additional counts derived in  Section~\ref{sec:FiniteIsometryGroup} for the various symmetry operations in $\mathcal{C}_{3h}$. In particular, one vertex is fixed by the improper rotation $S_3$ and no edges are fixed by $S_3$.

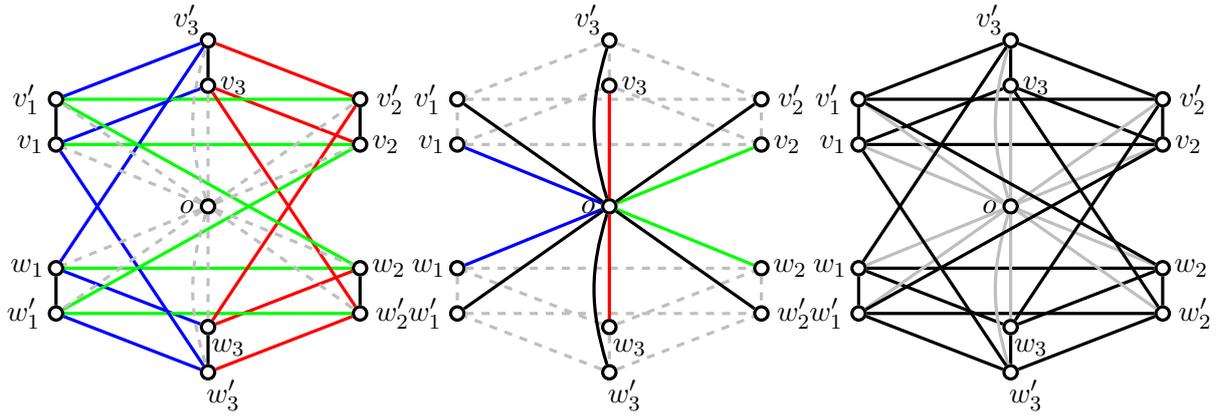
\begin{figure}[htp]
\centering
  \begin{tabular}{  c   }
  
 \begin{minipage}{.32\textwidth}
\begin{tikzpicture}[very thick,scale=1]
\tikzstyle{every node}=[circle, draw=black, fill=white, inner sep=0pt, minimum width=5pt];
\node [rectangle, draw=white, fill=white] (b) at (0.2,-2.5) {$w_3'$};
\node [rectangle, draw=white, fill=white] (b) at (-0.28,0) {$o$};

\path (-2,0.82) node (v1) [label = left: $v_{1}$] {} ;
\path (2,0.82) node (v2)[label = right: $v_{2}$]  {} ;
\path (0,1.6) node (v3) [label =  right: $v_{3}$] {} ;

\path (-2,1.42) node (v11) [label = left: $v_{1}'$] {} ;
\path (2,1.42) node (v22) [label =  right: $v_{2}'$] {} ;
\path (0,2.2) node (v33)[label = above left: $v_{3}'$]  {} ;

\path (-2,-0.82) node (w1) [label = left: $w_{1}$] {} ;
\path (2,-0.82) node (w2) [label =  right: $w_{2}$] {} ;
\path (0,-1.6) node (w3) [label =  below right: $w_{3}$] {} ;

\path (-2,-1.42) node (w11) [label =  left: $w_{1}'$] {} ;
\path (2,-1.42) node (w22) [label =  right: $w_{2}'$] {} ;
\path (0,-2.2) node (w33)  {} ;

\path (0,0) node (o) {} ;

\draw[green] (v1)  --  (v2);
\draw[red] (v3)  --  (v2);
\draw[blue] (v3)  --  (v1);

\draw[green] (v11)  --  (v22);
\draw[red] (v33)  --  (v22);
\draw[blue] (v33)  --  (v11);

\draw (v1)  --  (v11);
\draw (v2)  --  (v22);
\draw (v3)  --  (v33);

\draw[green] (w1)  --  (w2);
\draw[red] (w3)  --  (w2);
\draw[blue] (w3)  --  (w1);

\draw[green] (w11)  --  (w22);
\draw[red] (w33)  --  (w22);
\draw[blue] (w33)  --  (w11);

\draw (w1)  --  (w11);
\draw (w2)  --  (w22);
\draw (w3)  --  (w33);

\draw[blue] (v1)  --  (w33);
\draw[blue] (w1)  --  (v33);
\draw[red] (v3)  --  (w22);
\draw[red] (w3)  --  (v22);
\draw[green] (v2)  --  (w11);
\draw[green] (w2)  --  (v11);

\draw[dashed,lightgray] (o)  --  (v1);
\draw[dashed,lightgray] (o)  --  (v2);
\draw [dashed,lightgray](o)  --  (v3);
\draw[dashed,lightgray](o)  --  (v11);
\draw[dashed,lightgray] (o)  --  (v22);
 \path
(o) edge [dashed,lightgray,bend left=17] (v33);

\draw [dashed,lightgray](o)  --  (w1);
\draw [dashed,lightgray](o)  --  (w2);
 \path
(o) edge [dashed,lightgray,bend right=17] (w33);
\draw[dashed,lightgray] (o)  --  (w3);
\draw[dashed,lightgray] (o)  --  (w11);
\draw[dashed,lightgray] (o)  --  (w22);
\end{tikzpicture}
\end{minipage}

\begin{minipage}{.32\textwidth}
\begin{tikzpicture}[very thick,scale=1]
\tikzstyle{every node}=[circle, draw=black, fill=white, inner sep=0pt, minimum width=5pt];
\node [rectangle, draw=white, fill=white] (b) at (0.2,-2.5) {$w_3'$};
\node [rectangle, draw=white, fill=white] (b) at (-0.28,0) {$o$};
\path (-2,0.82) node (v1) [label = left: $v_{1}$] {} ;
\path (2,0.82) node (v2)[label = right: $v_{2}$]  {} ;
\path (0,1.6) node (v3) [label =  right: $v_{3}$] {} ;

\path (-2,1.42) node (v11) [label = left: $v_{1}'$] {} ;
\path (2,1.42) node (v22) [label =  right: $v_{2}'$] {} ;
\path (0,2.2) node (v33)[label = above left: $v_{3}'$]  {} ;

\path (-2,-0.82) node (w1) [label = left: $w_{1}$] {} ;
\path (2,-0.82) node (w2) [label =  right: $w_{2}$] {} ;
\path (0,-1.6) node (w3) [label =  below right: $w_{3}$] {} ;

\path (-2,-1.42) node (w11) [label =  left: $w_{1}'$] {} ;
\path (2,-1.42) node (w22) [label =  right: $w_{2}'$] {} ;
\path (0,-2.2) node (w33)  {} ;

\path (0,0) node (o) {} ;

\draw[dashed,lightgray] (v1)  --  (v2);
\draw[dashed,lightgray] (v3)  --  (v2);
\draw[dashed,lightgray] (v3)  --  (v1);

\draw[dashed,lightgray] (v11)  --  (v22);
\draw[dashed,lightgray] (v33)  --  (v22);
\draw[dashed,lightgray] (v33)  --  (v11);

\draw[dashed,lightgray]  (v1)  --  (v11);
\draw[dashed,lightgray]  (v2)  --  (v22);
\draw[dashed,lightgray]  (v3)  --  (v33);

\draw[dashed,lightgray] (w1)  --  (w2);
\draw[dashed,lightgray] (w3)  --  (w2);
\draw[dashed,lightgray] (w3)  --  (w1);

\draw[dashed,lightgray] (w11)  --  (w22);
\draw[dashed,lightgray] (w33)  --  (w22);
\draw[dashed,lightgray] (w33)  --  (w11);

\draw[dashed,lightgray]  (w1)  --  (w11);
\draw[dashed,lightgray]  (w2)  --  (w22);
\draw[dashed,lightgray]  (w3)  --  (w33);

\draw[blue] (o)  --  (v1);
\draw[green] (o)  --  (v2);
\draw [red](o)  --  (v3);
\draw(o)  --  (v11);
\draw (o)  --  (v22);
 \path
(o) edge [bend left=17] (v33);

\draw [blue](o)  --  (w1);
\draw [green](o)  --  (w2);
 \path
(o) edge [bend right=17] (w33);
\draw[red] (o)  --  (w3);
\draw (o)  --  (w11);
\draw (o)  --  (w22);
\end{tikzpicture}
\end{minipage}

\begin{minipage}{.32\textwidth}
\begin{tikzpicture}[very thick,scale=1]
\tikzstyle{every node}=[circle, draw=black, fill=white, inner sep=0pt, minimum width=5pt];
\node [rectangle, draw=white, fill=white] (b) at (0.2,-2.5) {$w_3'$};
\node [rectangle, draw=white, fill=white] (b) at (-0.28,0) {$o$};
\path (-2,0.82) node (v1) [label = left: $v_{1}$] {} ;
\path (2,0.82) node (v2)[label = right: $v_{2}$]  {} ;
\path (0,1.6) node (v3) [label =  right: $v_{3}$] {} ;

\path (-2,1.42) node (v11) [label = left: $v_{1}'$] {} ;
\path (2,1.42) node (v22) [label =  right: $v_{2}'$] {} ;
\path (0,2.2) node (v33)[label = above left: $v_{3}'$]  {} ;

\path (-2,-0.82) node (w1) [label = left: $w_{1}$] {} ;
\path (2,-0.82) node (w2) [label =  right: $w_{2}$] {} ;
\path (0,-1.6) node (w3) [label =  below right: $w_{3}$] {} ;

\path (-2,-1.42) node (w11) [label =  left: $w_{1}'$] {} ;
\path (2,-1.42) node (w22) [label =  right: $w_{2}'$] {} ;
\path (0,-2.2) node (w33)  {} ;

\path (0,0) node (o) {} ;

\draw (v1)  --  (v2);
\draw (v3)  --  (v2);
\draw (v3)  --  (v1);

\draw (v11)  --  (v22);
\draw (v33)  --  (v22);
\draw (v33)  --  (v11);

\draw (v1)  --  (v11);
\draw (v2)  --  (v22);
\draw (v3)  --  (v33);

\draw (w1)  --  (w2);
\draw (w3)  --  (w2);
\draw (w3)  --  (w1);

\draw(w11)  --  (w22);
\draw(w33)  --  (w22);
\draw (w33)  --  (w11);

\draw (w1)  --  (w11);
\draw (w2)  --  (w22);
\draw (w3)  --  (w33);

\draw[lightgray] (o)  --  (v1);
\draw[lightgray] (o)  --  (v2);
\draw [lightgray](o)  --  (v3);
\draw[lightgray] (o)  --  (v11);
\draw[lightgray] (o)  --  (v22);
 \path
(o) edge [lightgray,bend left=17] (v33);

\draw [lightgray](o)  --  (w1);
\draw [lightgray](o)  --  (w2);
 \path
(o) edge [lightgray,bend right=17] (w33);
\draw[lightgray] (o)  --  (w3);
\draw [lightgray](o)  --  (w11);
\draw [lightgray](o)  --  (w22);

\draw (v1)  --  (w33);
\draw (w1)  --  (v33);
\draw (v3)  --  (w22);
\draw (w3)  --  (v22);
\draw (v2)  --  (w11);
\draw (w2)  --  (v11);
\end{tikzpicture}
\end{minipage}
\end{tabular}
\caption{The underlying graph of the $3$-dimensional isostatic framework with $\mathcal{C}_{3h}$ symmetry described in Example~\ref{ex:3dexamp}. The induced framework colours are illustrated in the left and centre images. For clarity, on the right all edges incident to the cone vertex $o$ are shown in gray.}
\label{fig:3dexamp}
\end{figure}
\end{example}


\section{Sufficient conditions for isostatic realisations.}
\label{sec:furwork}
In the previous sections, we have established necessary conditions for symmetric and non-symmetric bar-joint frameworks to be isostatic in general normed linear spaces. These conditions included both over-all counts and subgraph counts on the number of vertices and edges (recall Theorem \ref{thm:maxwell}), as well as counts on the number of vertices and edges that are fixed by various symmetry operations of the framework (recall Corollary~\ref{maxwellcor}).
It is natural to ask whether these necessary conditions are also sufficient for the framework to be isostatic. To investigate these questions, we introduce the following notation:

\begin{definition}
A graph $G=(V,E)$ is called \emph{$(k,\ell)$-tight} if $|E|= k|V|-\ell$ and $|E'|\leq k|V'|-\ell$ for all subgraphs $G'=(V',E')$ of $G$ with at least two vertices.
\end{definition}

\subsection{Euclidean frameworks.}
For the Euclidean plane, isostatic generic bar-joint frameworks were characterised combinatorially by Laman's landmark result from 1970, which says that a generic bar-joint framework $(G,p)$ in $(\mathbb{R}^2,\|\cdot\|_2)$ is isostatic if and only if the graph $G$ is $(2,3)$-tight. (See also \cite{tay}, for example, for an alternative characterisation in terms of tree decompositions.) The analogous questions for Euclidean frameworks in dimensions $d\geq 3$ remain long-standing open problems in discrete geometry, although significant partial results have been obtained for the special classes of body-bar, body-hinge and molecular frameworks.

Using a similar approach as in Section~\ref{sec:symmetric},  it was shown in \cite{cfgsw,owen}, that a symmetric isostatic bar-joint framework in $(\mathbb{R}^2,\|\cdot\|_2)$ must not only be $(2,3)$-tight, but must also satisfy some restrictions on the number of vertices and edges that are fixed by various symmetry operations of the framework. In particular, these restrictions imply that there are only five possible non-trivial symmetry groups which allow for  an isostatic bar-joint framework in $(\mathbb{R}^2,\|\cdot\|_2)$. These are the rotational groups $\mathcal{C}_2$ and $\mathcal{C}_3$, the reflection group $\mathcal{C}_s$, and the dihedral groups $\mathcal{C}_{2v}$ and $\mathcal{C}_{3v}$. In higher dimensions, all symmetry groups are possible, although restrictions on the number of fixed structural components still apply.

It was shown in \cite{schulze,BS4} that for the groups $\mathcal{C}_2$, $\mathcal{C}_3$ and $\mathcal{C}_s$ in the Euclidean plane, Laman's conditions (i.e., $(2,3)$-tightness), together with the added necessary conditions on the number of fixed structural components, are also sufficient for a $\Gamma$-generic framework to be isostatic. For the dihedral groups $\mathcal{C}_{2v}$ and $\mathcal{C}_{3v}$, however, the analogous conjectures are still open.

Note that an infinitesimally rigid symmetric framework in $(\mathbb{R}^2,\|\cdot\|_2)$ need not have a spanning isostatic subframework with the same symmetry. Therefore, for $\Gamma$-generic frameworks, infinitesimal rigidity can in general not be characterised in terms of isostatic subframeworks. Characterisations of $\Gamma$-generic infinitesimally rigid frameworks have so far only been obtained for the groups $\Gamma=\mathcal{C}_2,\mathcal{C}_3$ and $\mathcal{C}_s$ in dimension $2$, where $\Gamma$ acts freely on the vertices \cite{schtan}. These results were established by the introduction of new symmetry-adapted rigidity matrices and an analysis of their corresponding matroids on group-labeled quotient graphs. Extensions of these results to body-bar and body-hinge frameworks with $\mathbb{Z}_2\times \cdots \times \mathbb{Z}_2$ symmetry in higher-dimensional spaces can be found in \cite{schtan2}.

\subsection{Non-Euclidean frameworks.}
For the non-Euclidean norms $\|\cdot \|_q$, $1\leq q \leq \infty$, $q\neq 2$, and for the polyhedral norms, analogues of Laman's theorem have  recently been established in \cite{kit-pow} and \cite{kitson}. Specifically, it was shown in \cite{kit-pow} that a well-positioned regular bar-joint framework $(G,p)$ in  $(\mathbb{R}^2,\|\cdot \|_q)$, $1\leq q \leq \infty$, $q\neq 2$, is isostatic if and only if the graph $G$ is $(2,2)$-tight. Moreover, it was shown in \cite{kitson} that if $\|\cdot\|_\P$ is any polyhedral norm on $\mathbb{R}^2$, then there exists a well-positioned isostatic bar-joint framework $(G,p)$ in $(\mathbb{R}^2,\|\cdot\|_\P)$ if and only if the graph $G$ is $(2,2)$-tight. Symmetric analogues of these results have not yet been established. However, for some polyhedral norms on $\mathbb{R}^2$ and their possible symmetry groups, we will offer some observations and conjectures in the next section (Section~\ref{subsec:obsconj}).

As far as higher-dimensional spaces are concerned, it should be pointed out that for some non-Euclidean normed spaces, the problem of finding a combinatorial characterisation of (non-symmetric) isostatic bar-joint frameworks in dimension $d\geq 3$ might be slightly more accessible than its Euclidean counterpart, because in cases where rotations are no longer isometries, complexities such as the double-banana graph \cite{W1} may no longer occur.

\subsection{Observations and conjectures for $(\mathbb{R}^2, \|\cdot\|_\P)$, where $\P$ is a quadrilateral.} \label{subsec:obsconj}

For a well-positioned $\Gamma$-symmetric bar-joint framework $(G,p)$ in a space of the form $(\mathbb{R}^2, \|\cdot\|_\P)$, where $\P$ is a polyhedron,  to be isostatic, $G$ must be $(2,2)$-tight (by Theorem~\ref{thm:maxwell}) and $G$ must satisfy the conditions in Section \ref{sec:FiniteIsometryGroup}. However, if the unit ball $\P$ is a quadrilateral, then we can easily establish some further necessary conditions for isostaticity. We begin by considering frameworks with reflectional or half-turn symmetry. We first need the following definition.

\begin{definition}
 Let $\|\cdot \|_\P$ be a polyhedral norm on $\bR^2$ for which the unit ball $\P$ is a quadrilateral and let  $(G,p)$ a framework in $(\bR^2,\|\cdot\|_\P)$ which is $\Gamma$-symmetric with respect to $\theta:\Gamma\to \Aut(G)$ and  $\tau:\Gamma\to \GL(\bR^2)$.
A symmetry operation $\gamma\in \Gamma$  {\em preserves the facets} of  $\P$ if $\tau(\gamma)F\in \{F,-F\}$  for each facet $F$ of $\P$. Otherwise, we say that $\gamma$  {\em swaps the facets} of $\P$.
\end{definition}

\begin{proposition}
\label{FixedVertexLemma}
\label{thm:cscounts}
 Let $\|\cdot \|_\P$ be a polyhedral norm on $\bR^2$ for which the unit ball $\P$ is a quadrilateral.
Let  $(G,p)$ be a well-positioned isostatic framework which is $\Gamma$-symmetric with respect to $\theta:\Gamma\to \Aut(G)$ and
 $\tau:\Gamma\to \GL(\bR^2)$. Then $G$ is $(2,2)$-tight, and the following conditions hold:
\begin{itemize}
\item[(i)] For $\Gamma=\mathcal{C}_s$: If $s$ preserves the facets of $\P$ then
$|V_{s}|=1$, $|E_{s}|=0$, the degree of the fixed vertex is at least 4, and every $\mathcal{C}_s$-symmetric $(2,2)$-tight subgraph $H$ of $G$ contains the fixed vertex.

\item[(ii)] For  $\Gamma=\mathcal{C}_2$:
 there does not exist a $\mathcal{C}_2$-symmetric $(2,2)$-tight subgraph $H$ of $G$ which has no vertices or edges fixed by $C_2$, and either  $|V_{2}|=1$, $|E_{2}|=0$, and the degree of the fixed vertex is at least 4, or  $|V_{2}|=0$ and $|E_{2}|=2$.
\end{itemize}
\end{proposition}

\proof
$(i)$ By Theorem~\ref{thm:maxwell}, $G$ is $(2,2)$-tight, and by Corollary~\ref{Reflection} $(ii)$, we have $|E_s|=0$. Note that since $(G,p)$ is a well-positioned isostatic framework,  $G$ is the union of two edge-disjoint spanning trees $T_1$ and $T_2$ which are induced by the framework colouring of $(G,p)$ \cite{kit-pow} (recall also Example~\ref{ex:2dexamp}). Both $T_1$ and $T_2$ are $\C_s$-symmetric with respect to $\theta$, since $s$ preserves the facets of $\P$. To see this, suppose without loss of generality that $vw$ is an edge of $T_1$, and let $F$ and $-F$ be the facets of the unit ball $\P$ corresponding to $T_1$, that is, $p(v)-p(w)$ is contained in the conical hull of either the facet $F$ or the facet $-F$.
Since $s$  preserves the facets of $\P$ and $p(s v)-p(s w)=\tau(s)(p(v)-p(w))$, it follows that
$p(s v)-p(s w)$ is contained in the conical hull of either $F$ or $-F$.
Thus the edge $s(vw)$ has the same framework colour as $vw$, and so $s(vw)\in E(T_1)$.

Suppose there exist two distinct vertices $v_0$ and $v_1$ which are fixed by  $s$. Then $T_1$ contains a simple path $P$ from $v_0$ to $v_1$. Since $T_1$ is $\C_s$-symmetric,  $s(P)$ is also a simple path in $T_1$ from $v_0$ to $v_1$.
Also, since $|E_s|=0$, $P\not=s(P)$. Thus $P\cup s(P)$ contains a cycle and is a subgraph of $T_1$. This is a contradiction and so there is at most one vertex in $G$ which is fixed by $s$.

The number of edges in each $\C_s$-symmetric spanning tree $T_i$ must be even since each edge $e\in E(T_i)$ has a distinct reflection $se$ which also lies in $T_i$. Each spanning tree contains $|V|-1$ many edges and so $|V|$ is odd. In particular, there must exist at least one vertex of $G$ which is fixed by $s$ since every vertex $v\in V$ which is not fixed by $s$ has a distinct reflection $sv$. As shown above, we have $|V_s|\leq 1$ and so we conclude that there exists exactly one vertex in $G$ which is fixed by $s$.

If the fixed vertex $v_0$ is adjacent to an edge $e\in E$ then $v_0$ is also adjacent to $se$.
Since $|E_s|=0$, $e\not=se$ and so $v_0$ must have even degree in $G$.
Since $T_1$ and $T_2$ are edge-disjoint spanning trees of $G$, $v_0$ is adjacent to some edge $e$ in $T_1$ and some edge $f$ in $T_2$ with $e\not=f$.
Since $v_0$ is fixed by $s$, $v_0$ is also adjacent to the edges $se$ in $T_1$ and $sf$ in $T_2$.
The edges $e,f,se,sf$ are necessarily distinct, and so $v_0$ has degree at least $4$ in $G$.

Finally, suppose $G$ has a $\mathcal{C}_s$-symmetric $(2,2)$-tight subgraph $H$ which has no vertices fixed by $s$. Then $H$ is the union of the two edge-disjoint spanning trees $H\cap T_1$ and $H\cap T_2$. Both of these trees have an even number of edges since $|E_s|=0$, but since $H$ has no vertex fixed by $s$, they also have an even number of vertices, a contradiction. (See also Figures~\ref{hcupsh} (a), (b)).

The proof of $(ii)$ is completely analogous to $(i)$. Simply recall Corollary \ref{2Fold} and note that the half-turn $C_2$  clearly preserves the facets of $\P$.
\endproof

\begin{conjecture} The conditions in Proposition~\ref{FixedVertexLemma}  are also sufficient for the existence of a well-positioned isostatic  framework $(G,p)$ in $(\bR^2,\|\cdot\|_\P)$ which is $\Gamma$-symmetric with respect to $\theta$ and $\tau$.
\end{conjecture}

A natural approach to prove this conjecture is via an inductive construction scheme similar to the ones used in \cite{schulze,BS4} for Euclidean symmetric frameworks. However, this would require some new symmetry-adapted Henneberg-type graph operations on $(2,2)$-tight graphs with a $\mathbb{Z}_2$-action, and for each of these operations, one would have to show that it preserves isostaticity by choosing appropriate geometric placements for the new vertices. 

\begin{conjecture}  Let $\|\cdot \|_\P$ be a polyhedral norm on $\bR^2$ for which the unit ball $\P$ is a quadrilateral.
Let $G$ be a finite simple graph, let  $\theta:\C_s\to \Aut(G)$ be an action of  $\C_s$, and
 $\tau:\C_s\to \GL(\bR^2)$ be a faithful group representation so that $s$ swaps the facets of $\P$. Then there exists $p$ such that $(G,p)$ is well-positioned and isostatic  in $(\bR^2,\|\cdot\|_\P)$ and $\C_s$-symmetric with respect to $\theta$ and $\tau$ if and only if $G$ is $(2,2)$-tight and $|E_s|=0$.
\end{conjecture}

Clearly, if $(G,p)$ is a well-positioned, isostatic and $\C_2$-symmetric framework in $(\bR^2,\|\cdot\|_\P)$, then $G$ is $(2,2)$-tight and $|E_s|=0$ (recall Corollary~\ref{Reflection} $(ii)$). The converse, however, remains open. Note that, as opposed to the case where $s$ preserves the facets of $\P$,  there is no restriction on the number of vertices of $G$ that are fixed by the reflection (recall Figure~\ref{fig:henstart} (b), for example). Moreover, $(G,p)$ could possibly have a $\mathcal{C}_s$-symmetric $(2,2)$-tight subgraph $H$ which has no vertices or edges fixed by $s$ in this case, as the example in Figure~\ref{hcupsh} (c) illustrates.

\begin{figure}[htp]
\begin{center}
\begin{tikzpicture}[very thick,scale=1]
\tikzstyle{every node}=[circle, draw=black, fill=white, inner sep=0pt, minimum width=5pt];

\draw[dashed,thin] (0,-0.6)  --  (0,0.6);

\path (-1.3,-0.2) node (p1)  {} ;
\path (-1.3,0.4) node (p2)  {} ;
\path (-0.6,-0.4) node (p3)  {} ;
\path (-0.6,0.3) node (p4)  {} ;

\path (1.3,-0.2) node (p11)  {} ;
\path (1.3,0.4) node (p22)  {} ;
\path (0.6,-0.4) node (p33)  {} ;
\path (0.6,0.3) node (p44)  {} ;

\draw[gray]  (p1)  --  (p2);
\draw (p1)  --  (p3);
\draw(p1)  --  (p4);
\draw[gray] (p2)  --  (p3);
\draw(p2)  --  (p4);
\draw[gray] (p3)  --  (p4);

\draw[gray]  (p11)  --  (p22);
\draw (p11)  --  (p33);
\draw(p11)  --  (p44);
\draw[gray] (p22)  --  (p33);
\draw(p22)  --  (p44);
\draw[gray] (p33)  --  (p44);

\draw(p3)  --  (p44);
\draw(p4)  --  (p33);
\node [rectangle, draw=white, fill=white] (b) at (0,-1.45) {(a)};
\end{tikzpicture}
\hspace{1.3cm}
\begin{tikzpicture}[very thick,scale=1]
\tikzstyle{every node}=[circle, draw=black, fill=white, inner sep=0pt, minimum width=5pt];

\path (-1.3,-0.2) node (p1)  {} ;
\path (-1.3,0.4) node (p2)  {} ;
\path (-0.6,-0.4) node (p3)  {} ;
\path (-0.6,0.3) node (p4)  {} ;

\path (1.3,0.2) node (p11)  {} ;
\path (1.3,-0.4) node (p22)  {} ;
\path (0.6,0.4) node (p33)  {} ;
\path (0.6,-0.3) node (p44)  {} ;

\draw[gray]  (p1)  --  (p2);
\draw (p1)  --  (p3);
\draw(p1)  --  (p4);
\draw[gray] (p2)  --  (p3);
\draw(p2)  --  (p4);
\draw[gray] (p3)  --  (p4);

\draw[gray]  (p11)  --  (p22);
\draw (p11)  --  (p33);
\draw(p11)  --  (p44);
\draw[gray] (p22)  --  (p33);
\draw(p22)  --  (p44);
\draw[gray] (p33)  --  (p44);

\draw(p3)  --  (p44);
\draw(p4)  --  (p33);
\node [rectangle, draw=white, fill=white] (b) at (0,-1.45) {(b)};
\end{tikzpicture}
\hspace{1.3cm}
\begin{tikzpicture}[very thick,scale=1]
\tikzstyle{every node}=[circle, draw=black, fill=white, inner sep=0pt, minimum width=5pt];

\draw[dashed,thin] (-0.8,-0.8)  --  (0.8,0.8);

\path (-1.1,0) node (p1)  {} ;
\path (-1,0.9) node (p2)  {} ;
\path (-0.4,0) node (p3)  {} ;
\path (-0.3,0.6) node (p4)  {} ;

\path (0,-1.1) node (p11)  {} ;
\path (0.9,-1) node (p22)  {} ;
\path (0,-0.4) node (p33)  {} ;
\path (0.6,-0.3) node (p44)  {} ;

\draw[gray]  (p1)  --  (p2);
\draw (p1)  --  (p3);
\draw(p1)  --  (p4);
\draw[gray] (p2)  --  (p3);
\draw(p2)  --  (p4);
\draw[gray] (p3)  --  (p4);

\draw (p11)  --  (p22);
\draw [gray] (p11)  --  (p33);
\draw[gray] (p11)  --  (p44);
\draw (p22)  --  (p33);
\draw[gray] (p22)  --  (p44);
\draw (p33)  --  (p44);

\draw (p3)  --  (p44);
\draw[gray](p4)  --  (p33);
\node [rectangle, draw=white, fill=white] (c) at (0.4,-1.45) {(c)};
\end{tikzpicture}
     \end{center}
     \caption{(a),(b) Examples of $\mathcal{C}_s$- and $\mathcal{C}_2$-symmetric frameworks in $(\bR^2,\|\cdot\|_\infty)$ which are not isostatic. (c) A $\mathcal{C}_s$-symmetric isostatic framework in $(\bR^2,\|\cdot\|_\infty)$.}
\label{hcupsh}
\end{figure}
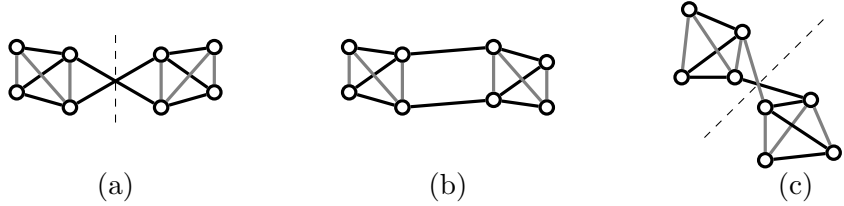

For the remaining symmetry groups in dimension $2$ which are possible for a quadrilateral unit ball $\P$, i.e., for the groups $\mathcal{C}_4$, $\mathcal{C}_{2v}$ and $\mathcal{C}_{4v}$, we propose the following conjecture:

\begin{conjecture}
\label{thm:cscounts}
Let $\|\cdot \|_\P$ be a polyhedral norm on $\bR^2$ for which the unit ball $\P$ is a quadrilateral.
Let $G$ be a finite simple graph, $\Gamma$ be a group, $\theta:\Gamma\to \Aut(G)$ be an action of  $\Gamma$, and
$\tau:\Gamma\to \GL(\bR^2)$ be a faithful group representation. 
The following are equivalent:
\begin{itemize}
\item[(A)] There exists $p$ such that $(G,p)$ is well-positioned and isostatic  in $(\bR^2,\|\cdot\|_\P)$ and $\Gamma$-symmetric with respect to $\theta$ and $\tau$.

\item[(B)] \begin{itemize}

\item[(i)] For $\Gamma=\C_4=\langle  r  \rangle$: $G$ is $(2,2)$-tight, there does not exist a $\langle r^2 \rangle$-symmetric $(2,2)$-tight subgraph $H$ of $G$ which has no vertices or edges fixed by $r^2$, and either  $|V_{r^2}|=1$, $|E_{r^2}|=0$, or  $|V_{r^2}|=0$ and $|E_{r^2}|=2$.

\item[(ii)] For $\Gamma=\mathcal{C}_{2v}=\langle s, r\rangle$, where $s$ is a reflection  which preserves the facets of $\P$, and  $r$ is the half-turn: $G$ is $(2,2)$-tight, for every non-trivial element $\gamma\in \Gamma$, there does not exist a $\langle \gamma \rangle$-symmetric $(2,2)$-tight subgraph $H$ of $G$ which has no vertices or edges fixed by $\gamma$, $|V_{\gamma}|=1$ and $|E_{\gamma}|=0$.

\item[(iii)] For $\Gamma=\mathcal{C}_{2v}=\langle s, r\rangle$, where $s$ is a reflection  which swaps the facets of $\P$, and  $r$ is the half-turn: $G$ is $(2,2)$-tight,  there does not exist a $\langle r \rangle$-symmetric $(2,2)$-tight subgraph $H$ of $G$ which has no vertices or edges fixed by $r$, $|E_{s}|=|E_{r\cdot s}|= 0$, and either  $|V_{r}|=1$, $|E_{r}|=0$, or  $|V_{r}|=0$ and $|E_{r}|=2$.

\item[(iv)] For $\Gamma=\mathcal{C}_{4v}=\langle s, r \rangle$, where $s$ is a reflection  which preserves the facets of $\P$, and  $r$ is a $4$-fold rotation: $G$ is $(2,2)$-tight, for $\gamma\in \{s, r^2,r^2\cdot s\}$, there does not exist a $\langle \gamma \rangle$-symmetric $(2,2)$-tight subgraph $H$ of $G$ which has no vertices or edges fixed by $\gamma$, $|E_{s}|=|E_{r^2\cdot s}|= 0$, and either  $|V_{r^2}|=1$, $|E_{r^2}|=0$, or  $|V_{r^2}|=0$ and $|E_{r^2}|=2$.
\end{itemize}
\end{itemize}
\end{conjecture}

As before, using the results of Sections~\ref{sec:basicdef} and \ref{sec:FiniteIsometryGroup}, it is easy to see that (A) implies (B). In particular, as shown in Figure~\ref{hcupsh1}, it is easy to construct non-isostatic frameworks which satisfy all the conditions in Conjecture~\ref{thm:cscounts}, except for the condition on the existence of a symmetric $(2,2)$-tight subgraph which has no vertices or edges fixed by a reflection or a half-turn. All the converse directions  of Conjecture~\ref{thm:cscounts}, however, remain open.

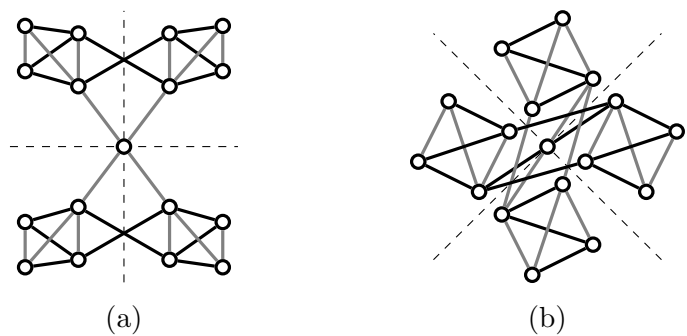
\begin{figure}[htp]
\begin{center}
\begin{tikzpicture}[very thick,scale=1]
\tikzstyle{every node}=[circle, draw=black, fill=white, inner sep=0pt, minimum width=5pt];

\draw[dashed,thin] (0,-1.8)  --  (0,1.8);
\draw[dashed,thin] (-1.5,0)  --  (1.5,0);

\path (-1.3,1) node (p1)  {} ;
\path (-1.3,1.6) node (p2)  {} ;
\path (-0.6,0.8) node (p3)  {} ;
\path (-0.6,1.5) node (p4)  {} ;

\path (1.3,1) node (p11)  {} ;
\path (1.3,1.6) node (p22)  {} ;
\path (0.6,0.8) node (p33)  {} ;
\path (0.6,1.5) node (p44)  {} ;

\draw[gray]  (p1)  --  (p2);
\draw (p1)  --  (p3);
\draw(p1)  --  (p4);
\draw[gray] (p2)  --  (p3);
\draw(p2)  --  (p4);
\draw[gray] (p3)  --  (p4);

\draw[gray]  (p11)  --  (p22);
\draw (p11)  --  (p33);
\draw(p11)  --  (p44);
\draw[gray] (p22)  --  (p33);
\draw(p22)  --  (p44);
\draw[gray] (p33)  --  (p44);

\draw(p3)  --  (p44);
\draw(p4)  --  (p33);

\path (-1.3,-1) node (p1u)  {} ;
\path (-1.3,-1.6) node (p2u)  {} ;
\path (-0.6,-0.8) node (p3u)  {} ;
\path (-0.6,-1.5) node (p4u)  {} ;

\path (1.3,-1) node (p11u)  {} ;
\path (1.3,-1.6) node (p22u)  {} ;
\path (0.6,-0.8) node (p33u)  {} ;
\path (0.6,-1.5) node (p44u)  {} ;

\draw[gray]  (p1u)  --  (p2u);
\draw (p1u)  --  (p3u);
\draw(p1u)  --  (p4u);
\draw[gray] (p2u)  --  (p3u);
\draw(p2u)  --  (p4u);
\draw[gray] (p3u)  --  (p4u);

\draw[gray]  (p11u)  --  (p22u);
\draw (p11u)  --  (p33u);
\draw(p11u)  --  (p44u);
\draw[gray] (p22u)  --  (p33u);
\draw(p22u)  --  (p44u);
\draw[gray] (p33u)  --  (p44u);

\draw(p3u)  --  (p44u);
\draw(p4u)  --  (p33u);

\path (0,0) node (c)  {} ;
\draw[gray](c)  --  (p3);
\draw[gray](c)  --  (p33);
\draw[gray](c)  --  (p3u);
\draw[gray](c)  --  (p33u);

\node [rectangle, draw=white, fill=white] (b) at (0,-2.3) {(a)};
\end{tikzpicture}
\hspace{2cm}
\begin{tikzpicture}[very thick,scale=1]
\tikzstyle{every node}=[circle, draw=black, fill=white, inner sep=0pt, minimum width=5pt];

\draw[dashed,thin] (-1.5,-1.5)  --  (1.5,1.5);

\draw[dashed,thin] (1.5,-1.5)  --  (-1.5,1.5);

\path (-0.6,1.3) node (p1)  {} ;
\path (0.6,0.9) node (p2)  {} ;
\path (-0.2,0.5) node (p3)  {} ;
\path (0.2,1.7) node (p4)  {} ;

\draw (p1)  --  (p2);
\draw[gray] (p1)  --  (p3);
\draw (p1)  --  (p4);
\draw[gray] (p2)  --  (p4);
\draw[gray] (p3)  --  (p4);
\draw (p3)  --  (p2);

\path (0.6,-1.3) node (p11)  {} ;
\path (-0.6,-0.9) node (p22)  {} ;
\path (0.2,-0.5) node (p33)  {} ;
\path (-0.2,-1.7) node (p44)  {} ;

\draw (p11)  --  (p22);
\draw[gray] (p11)  --  (p33);
\draw (p11)  --  (p44);
\draw[gray] (p22)  --  (p44);
\draw[gray] (p33)  --  (p44);
\draw (p33)  --  (p22);

\path (1.3,-0.6) node (p11l)  {} ;
\path (0.9,0.6) node (p22l)  {} ;
\path (0.5,-0.2) node (p33l)  {} ;
\path (1.7,0.2) node (p44l)  {} ;

\draw [gray](p11l)  --  (p22l);
\draw (p11l)  --  (p33l);
\draw [gray](p11l)  --  (p44l);
\draw (p22l)  --  (p44l);
\draw(p33l)  --  (p44l);

\draw [gray](p33l)  --  (p22l);

\path (-1.3,0.6) node (p11r)  {} ;
\path (-0.9,-0.6) node (p22r)  {} ;
\path (-0.5,0.2) node (p33r)  {} ;
\path (-1.7,-0.2) node (p44r)  {} ;

\draw [gray](p11r)  --  (p22r);
\draw (p11r)  --  (p33r);
\draw [gray](p11r)  --  (p44r);
\draw (p22r)  --  (p44r);
\draw (p33r)  --  (p44r);
\draw [gray](p33r)  --  (p22r);

\path (0,0) node (c)  {} ;
\draw [gray](c)  --  (p2);
\draw [gray](c)  --  (p22);
\draw (c)  --  (p22l);
\draw (c)  --  (p22r);
\draw [gray](p3)  --  (p22);
\draw [gray](p2)  --  (p33);
\draw (p33r)  --  (p22l);
\draw (p22r)  --  (p33l);
\node [rectangle, draw=white, fill=white] (b) at (0,-2.3) {(b)};
\end{tikzpicture}
\end{center}
\caption{$\mathcal{C}_{2v}$-symmetric frameworks in $(\bR^2,\|\cdot\|_\infty)$ which are not isostatic due to the existence of a $(2,2)$-tight subgraph which has no vertices or edges fixed by the reflection (corresponding to the vertical mirror line) (a) and the half-turn (b), respectively. Based on these examples it is easy to construct similar examples of non-isostatic frameworks for the groups $\mathcal{C}_{4}$ and $\mathcal{C}_{4v}$.}
\label{hcupsh1}
\end{figure}

\end{document}